\documentclass[11pt, twoside,reqno]{amsart}

\usepackage[top=1in,bottom=1in,left=1.2in,right=1.2in]{geometry}
\usepackage{amsmath,amssymb,amsthm,amsfonts,enumerate, mathrsfs, mathtools, bm}

\usepackage{appendix}
\usepackage{relsize}
\usepackage{stmaryrd}
\usepackage{esint}
\usepackage{xcolor}
\usepackage{graphicx}
\usepackage[numbers]{natbib}
\usepackage{tikz}

\usetikzlibrary{patterns,positioning,arrows,shapes,trees}

\usepackage[
colorlinks=true,
linkcolor=blue,
filecolor=blue,
urlcolor=red,
citecolor=magenta]{hyperref}
\usepackage[nameinlink]{cleveref}

\newcounter{counterConstant} 
\newcommand{\const}[1]{
	\addtocounter{counterConstant}{1}
	\edef#1{\arabic{counterConstant}}
}

\numberwithin{equation}{section}
\def\theequation{\arabic{section}.\arabic{equation}}

\newtheorem{theorem}{Theorem}[section]
\newtheorem{lemma}[theorem]{Lemma}
\newtheorem{remark}[theorem]{Remark}
\newtheorem{definition}[theorem]{Definition}
\newtheorem{proposition}[theorem]{Proposition}
\newtheorem{corollary}[theorem]{Corollary}
\newtheorem{assumption}{Assumption}

\def\cD{{\mathcal D}}
\def\cE{{\mathcal E}}
\def\cF{{\mathcal F}}

\def\cH{{\mathcal H}}

\def\cJ{{\mathcal J}}

\def\cM{{\mathcal M}}

\def\mE{{\mathbb E}}

\def\mP{{\mathbb P}}

\def\mR{{\mathbb R}}
\def\mS{{\mathbb S}}

\def\mW{{\mathbb W}}

\def\bE{{\mathbf E}}

\def\geq{\geqslant}
\def\leq{\leqslant}

\def\esssup{\mathop{\mathrm{ess\,sup}}}

\def\1{{\mathbf{1}}}

\def\d{\text{\rm{d}}}
\def\e{\mathrm{e}}

\def\div{\mathord{{\rm div}\,}}

 \usepackage{fancyhdr}
{} 
\begin{document}
	
\title{Diffusions with critical divergence-free drifts: Varadhan asymptotics and path-law singularity}

\author{Jintao Feng and Guohuan Zhao}

\address{State Key Laboratory of Mathematical Sciences, Academy of Mathematics and Systems Science, CAS, Beijing, 100190, China}
\email{fengjintao@amss.ac.cn}

\address{State Key Laboratory of Mathematical Sciences, Academy of Mathematics and Systems Science, CAS, Beijing, 100190, China}
\email{gzhao@amss.ac.cn}

%\thanks{The research of Guohuan Zhao is supported by the National Key Research and Development Program of China (No. 2024YFA1013503) and the National Natural Science Foundation of China (Grant No. 12271352).}

\begin{abstract}
We study the Euclidean Varadhan formula (EVF) and path laws for diffusions with critical divergence-free drifts in dimension $d\geq2$. For drifts given by smooth bounded skew-symmetric stream matrices, we prove a quantitative lower bound under a sector condition with a lower-order term. Its constants depend only on the dimension and the sector condition data. By approximation, the locally uniform Euclidean lower bound holds for all divergence-free $L^\infty_t{\rm BMO}^{-1}_x$ drifts. However, a time-dependent counterexample shows that EVF can fail even for a drift in $L^\infty_tL^d_x$. A complementary upper bound is given under a structural decomposition condition. In particular, for every time-independent $L^{d,\infty}$ drift (without any smallness assumption), and the decoupled diffusion associated with the two-dimensional vorticity equation, the EVF holds. Finally, we construct a time-independent $L^d$ drift whose SDE has a unique strong solution from the origin and whose path law is mutually singular with the corresponding driftless Brownian law. The same diffusion satisfies EVF, showing that Brownian short-time endpoint costs can persist despite singularity of the path law.
\end{abstract}
	
\maketitle

\setcounter{tocdepth}{1}
\tableofcontents
	
	

\section{Introduction and main results}

\subsection{Motivations}

Let $d\geq2$.  Consider the time-inhomogeneous diffusion formally given by
\begin{equation}\label{eq:sde}
\d X_t=u(t,X_t)\,\d t+\sqrt2\,\d W_t,
\quad X_0=x,
\end{equation}
where the drift is divergence-free and may be distribution-valued.  Its formal generator is
\[
L_t:=\Delta+u(t,\cdot)\cdot\nabla.
\]

Equation \eqref{eq:sde} describes the motion of a stochastic Lagrangian particle transported by an incompressible velocity field.  Its transition operators describe the corresponding advection--diffusion equation and arise in passive-scalar transport, convection-enhanced diffusion, and stochastic representations of incompressible fluid equations; see, for example, \cite{osada1985vortex,fannjiang1994convection, constantin2008stochastic,GP24weak,HZ25SDEs,RZZ26mckean} and the references therein.  Since the velocity fields in these models are often irregular, we ask which Brownian features survive at scale-critical regularity.  We study both the short-time cost of reaching a given endpoint and absolute continuity of the whole path law.  Our results show that these two properties can behave differently for the same critical drift.

The critical spaces can be distinguished by the diffusive scaling.  If $X_t^{(\lambda)}:=\lambda^{-1}X_{\lambda^2t}$, then the rescaled drift is
\[
u_\lambda(t,x):=\lambda u(\lambda^2t,\lambda x),
\quad
\|u_\lambda\|_{L^\infty_tL^p_x}
=\lambda^{1-d/p}\|u\|_{L^\infty_tL^p_x}.
\]
Thus, $p>d$ is subcritical, whereas $p=d$ is invariant.  Likewise, the Lorentz space $L^{d,\infty}$ is invariant under this scaling and contains $L^d$.  In the divergence-free setting, it is also natural to write
\begin{equation}\label{eq:def_u}
    u_j=\sum_{i=1}^d\partial_iU_{ij}, \quad U^{\mathsf T}=-U.
\end{equation}
We call $U$ a skew-symmetric stream matrix.  The larger distributional class $u\in L^\infty_t{\rm BMO}^{-1}_x$ is again invariant under the same scaling.   We refer the reader to Appendix \ref{sec:aux} for the definitions of ${\rm BMO}$ and ${\rm BMO}^{-1}$ and the associated critical embeddings.  In particular, Lemma \ref{lem:hodge} shows that every divergence-free $u\in L^{d,\infty}$ admits a ${\rm BMO}$ stream matrix.

The first Brownian feature we consider is the short-time cost of the transition density.  When $u=0$, the transition density of $\sqrt2W$ is
\[
q_t(x,y):=(4\pi t)^{-d/2}
\exp\Big(-\frac{|x-y|^2}{4t}\Big).
\]
Consequently, the Euclidean Varadhan formula takes the form
\begin{equation}\label{eq:var}
\lim_{t\downarrow0}t\log p_{0,t}(x,y)
=-\frac{|x-y|^2}{4}.\tag{EVF}
\end{equation}

Pointwise Varadhan asymptotics are classical under stronger regularity hypotheses.  For uniformly elliptic operators with regular coefficients, the formula goes back to Varadhan \cite{varadhan1967behavior}.  In the symmetric divergence-form setting, Norris \cite{norris1997heat} proved the pointwise heat kernel asymptotic for the Dirichlet form $\cE(f,f):=\int_M|\nabla f|^2\,\d m$ on Lipschitz Riemannian manifolds. For singular Euclidean drifts, Chen--Fang--Zhang \cite{CFZ19small} proved a small-time large-deviation principle and a pointwise Varadhan formula in dimension $d\geq3$ for the time-homogeneous generator $\frac12\Delta+\mu\cdot\nabla$, where the components of the measure $\mu$ belong to a Kato class.

An another line of work concerns set-to-set asymptotics. Hino--Ram\'{i}rez \cite{hinoRamirez2003}, and subsequently Ariyoshi--Hino \cite{ariyoshiHino2005}, obtained intrinsic-distance formulas for broad classes of symmetric strongly local Dirichlet forms. Hino--Matsuura \cite{hinoMatsuura2018} extended these formulas to fixed nonsymmetric forms under coercivity, sector, and form-smallness assumptions.  In their Euclidean example for $d\geq3$, the first-order coefficients belong to $L^d+L^\infty$, but not general $L^{d,\infty}$ drifts.  Their conclusion is in integral form rather than pointwise, and their first-order perturbations are measurable vector fields rather than general distribution-valued ${\rm BMO}^{-1}$ drifts.

At the critical distributional regularity considered here, Qian--Xi \cite{QX19parabolic2} constructed a fundamental solution (transition density) $p_{s,t}(x,y)$ satisfying the two-sided Aronson bounds
\begin{equation}\label{eq:aronson}
 \frac{(t-s)^{-d/2}}{C}\exp\Big(-\frac{C|x-y|^2}{t-s}\Big)
\leq p_{s,t}(x,y)
\leq C (t-s)^{-d/2}
\exp\Big(-\frac{|x-y|^2}{C(t-s)}\Big). \tag{AE}
\end{equation}
Here $C$ depends only on $d$ and $\|U\|_{L^\infty_t{\rm BMO}_x}$.
These bounds show that the Gaussian scale survives, but their unspecified exponential constants do not determine the limit in \eqref{eq:var}. And the aforementioned pointwise results for Varadhan formula do not cover general time-dependent ${\rm BMO}^{-1}$ drifts, while the broader nonsymmetric form result is in integral form and time-independent.  This leaves our first question:
\begin{enumerate}
    \item [({\bf Q}$_1$):] Does \eqref{eq:var} hold for the transition density associated with \eqref{eq:sde} when the divergence-free drift belongs to the critical space $L^\infty_t {\rm BMO}^{-1}_x$?
\end{enumerate}

One probabilistic approach to short-time density asymptotics is to use a change of measure on path space. For instance, Azencott \cite{azencott1984density} combined the Laplace method with the Cameron--Martin--Girsanov formula to obtain small-time density expansions.

In the subcritical regime $u\in L^\infty_tL^p_x$ with $p>d$, we have
\begin{equation}\label{eq:sub}
\bE_x\int_0^t|u(s,\sqrt{2}W_s)|^2\,\d s
\leq 
C t^{1-d/p}\|u\|_{L^\infty_tL^p_x}^2.
\end{equation}
Here $C$ depends only on $d,p$. Khasminskii's lemma and Girsanov's theorem show that, on every finite time interval, the solution law is absolutely continuous with respect to the law of $\sqrt2W$.  Thus, in this subcritical regime, the diffusion law can be compared directly with the Brownian law.  At the critical exponent, the right-hand side of \eqref{eq:sub} no longer yields such a comparison.

We point out that this difficulty is separate from well-posedness: criticality of the drift does not necessarily destroy well-posedness of \eqref{eq:sde}.  For time-homogeneous $L^d$ drifts, Krylov \cite{krylov2021stochastic2} first proved the existence and uniqueness of a strong solution to \eqref{eq:sde}. Weak well-posedness of \eqref{eq:sde} for more general time-dependent drifts, including $L^\infty_tL^{d}_x$ drifts, was studied by R\"ockner--Zhao \cite{RZ13weak,RZ25sde} and Kinzebulatov--Vafadar \cite{KV25critical}. However, these results do not determine whether the solution law is absolutely continuous with respect to the law of the corresponding Brownian motion.  It is therefore natural to ask:
\begin{enumerate}
    \item [({\bf Q}$_2$):] Is the law of the diffusion path absolutely continuous with respect to the law of the corresponding Brownian motion, at least, when the drift belongs to the critical space $L^d_x$?
\end{enumerate}

\subsection{Main results}

In this work, we consider a locally integrable skew-symmetric stream matrix $U(t)$, and the vector field $u$ is given by \eqref{eq:def_u}. Thus, $u$ is divergence-free in distributional sense. Spatial integrals without a specified domain are over $\mR^d$. Define the forms corresponding to the operator $L_t$, initially for real $f,g\in C_c^\infty(\mR^d)$, by
\begin{equation}\label{eq:local-forms}
\begin{aligned}
\cE_t(f,g):=&\cE^s(f,g)+\cE_t^a(f,g) \mbox{ with }\\
\cE^s(f,g):=\int\nabla f\cdot & \nabla g, \quad \cE_t^a(f,g):=\int U(t)\nabla f\cdot\nabla g.\\
\end{aligned}
\end{equation}
We use the same notation for their continuous extensions to $H^1(\mR^d)$. For ${\rm BMO}$ stream matrices, these extensions are defined via Hardy--BMO duality (see \eqref{eq:bmo}--\eqref{eq:hbmo} below). For another skew-symmetric matrix $A$, we denote the corresponding antisymmetric form by $\cE_A^a$.

For $U\in C_b^\infty(\mR_+\times\mR^d;\mR^{d\times d})$, let $p_{s,t}(x,y)$ be the heat kernel of $L_t=\Delta+u(t)\cdot\nabla$, where $u=\div U$. We first state the lower and upper estimates for smooth bounded stream matrices. Both Theorems \ref{thm:lower} and \ref{thm:upper} are proved independently of the Aronson bounds \eqref{eq:aronson}. Their constants are controlled solely by the dimension and the data in the sector condition and the structural decomposition, respectively; no Aronson constant enters either estimate. 

\begin{assumption}\label{ass:sector}
Let $T>0$. We say that $\cE^a_t$ satisfies the \emph{sector condition} with a time-dependent lower-order term if there are $M\geq0$ and a nonnegative $m\in L^2(0,T)$ such that, for almost every $t\in(0,T)$ and all real $f,g\in C_c^\infty(\mR^d)$,
\begin{equation}\label{eq:sector}
|\cE_t^a(f,g)|\leq\|\nabla f\|_2
\bigl(M\|\nabla g\|_2+m(t)\|g\|_2\bigr).
\tag{SC}
\end{equation}
\end{assumption}
By skew symmetry, the same inequality holds with $f$ and $g$ interchanged. The case $m=0$ includes ${\rm BMO}$ stream matrices. More generally, a drift $u=a+b$ with divergence-free $a\in L^\infty_t{\rm BMO}^{-1}_x$ and $b\in L^2_tL^\infty_x$ on $(0,T)$ satisfies \eqref{eq:sector} (see Lemma \ref{lem:sector-sum}).

\begin{theorem}[Lower bound]\label{thm:lower}
Let $d\geq2$ and let $U\in C_b^\infty(\mR_+\times\mR^d;\mR^{d\times d})$ be skew-symmetric. Suppose that Assumption \ref{ass:sector} holds, and set $N:=\|m\|_{L^2(0,T)}$. There is a constant $C$ such that, for all $x,y\in\mR^d$ and $0<t<T\wedge \frac{1}{4}$,
\begin{equation}\label{eq:lower}
p_{0,t}(x,y)\geq t^{-d/2}
\exp\left\{-\frac{|x-y|^2}{4t(1-\sqrt t)^2}
-\frac{C}{\sqrt t}\right\}.
\end{equation}
Here $C$ depends only on $d,M,N$, and not on any other norm of $U$ or its derivatives.
\end{theorem}

Theorem \ref{thm:lower} has the following consequence for BMO drifts; its proof is given at the end of Section \ref{sec:lower}.

\begin{corollary}\label{cor:bmo-lower}
Let $d\geq2$ and let $u\in L^\infty_t{\rm BMO}^{-1}_x$ be divergence-free. The associated heat kernel $p$ constructed in Theorem \ref{thm:feller} satisfies, for every $x,y\in\mR^d$,
\[
\liminf_{t\downarrow0}t\log p_{0,t}(x,y)
\geq-\frac{|x-y|^2}{4}.
\]
The lower bound holds locally uniformly in $(x,y)$.
\end{corollary}

However, this lower bound need not be accompanied by the Euclidean upper bound, even for a time-dependent drift uniformly bounded in the critical spatial space $L^d$.

\begin{theorem}\label{thm:counter}
For every $d\geq2$, there is a divergence-free vector field $u\in L^\infty_tL^d_x\subseteq L^\infty_t{\rm BMO}^{-1}_x$ such that every spatial slice is smooth and compactly supported, and its heat kernel satisfies
\[
\limsup_{t\downarrow0}t\log p_{0,t}(0,e_1)
\geq-\frac19>-\frac14.
\]
In particular, the pointwise Euclidean Varadhan upper bound fails.
\end{theorem}

%Thus, the $L^\infty_t{\rm BMO}^{-1}_x$ condition implies the Euclidean lower bound but does not determine the short-time exponent. The proof in Section \ref{sec:counter} yields a finite subsequential limit at $(0,e_1)$ different from the Brownian value $-1/4$.

To obtain the matching upper bound, we assume that the drift admits the following decomposition.
\begin{assumption}\label{ass:decomp}
Let $T>0$. There are a divergence-free vector field $v$ and a nonnegative, time-independent $F\in L^{d,\infty}(\mR^d)$ such that, for every $\sigma>0$, one can write
\begin{equation}\label{eq:decomp}
\begin{aligned}
u=v+a_\sigma+b_\sigma,& \quad |v(t,x)|\leq F(x),\\
\esssup_{0<t<T}\|a_\sigma(t)\|_{{\rm BMO}^{-1}}\leq\sigma,
&\quad
\beta_\sigma(t):=\|b_\sigma(t)\|_\infty\in L^1(0,T), 
\end{aligned}\tag{DC}
\end{equation}
where the fields $a_\sigma,b_\sigma$ are divergence-free. The same $v$ and $F$ are used for every $\sigma$.
\end{assumption}

\begin{theorem}[Upper bound]\label{thm:upper}
Let $d\geq2$, and $U\in C_b^\infty(\mR_+\times\mR^d;\mR^{d\times d})$ be skew-symmetric. Suppose that $u=\div U$ satisfies Assumption \ref{ass:decomp} on $(0,T)$, and put $N_F:=\|F\|_{L^{d,\infty}}$. For each $\delta>0$, there are $\sigma_\delta>0$, depending only on $d,\delta$, and $K_\delta\geq0$, depending only on $d,\delta,N_F$, such that, with
\begin{equation}\label{eq:upper-data}
\begin{aligned}
&c_\delta(t):=K_\delta+\beta_{\sigma_\delta}(t), \quad
e_\delta(t):=\delta+\int_0^t c_\delta(s)\,\d s, \\
&B_\delta(t):=\Big(1+\frac{\delta^2}{2}\Big)t+3(\delta+2)^2t^2,
\end{aligned}
\end{equation}
the kernel satisfies, for every $x,y\in\mR^d$ and $0<t<T\wedge \frac{1}{4}$,
\begin{equation}\label{eq:weighted-opt}
p_{0,t}(x,y)\leq Ct^{-d}
\exp\left\{-\frac{(|x-y|-e_\delta(t))_+^2}{4B_\delta(t)}\right\}.
\end{equation}
Here $C$ depends only on $d$. The estimates are uniform over smooth coefficients whose decompositions share the same envelope $F$ and bounds $\beta_\sigma(t)$.
\end{theorem}

For $L^\infty_t {\rm BMO}^{-1}_x$ drifts, Appendix \ref{sec:approx} uses \eqref{eq:aronson} to construct the heat kernel and establish convergence of its smooth approximations. Combining this approximation with the lower and upper estimates, we obtain the following result.

\begin{corollary}[Euclidean Varadhan formula]\label{cor:classes}
Let $d\geq2$ and let $u\in L^\infty_t{\rm BMO}^{-1}_x$ be divergence-free. Suppose that Assumption \ref{ass:decomp} holds. Then $u$ has a heat kernel satisfying the Euclidean Varadhan formula \eqref{eq:var}, locally uniformly in $x,y\in\mR^d$. In particular, this conclusion holds for every time-independent, divergence-free drift $u\in L^{d,\infty}(\mR^d;\mR^d)$.
\end{corollary}

No smallness condition is imposed on the weak-$L^d$ norm. In dimension two, the corollary includes the point-vortex field
\[
u_\lambda(x):=\lambda\frac{(-x_2,x_1)}{|x|^2},
\quad x\neq 0,~\lambda\neq 0.
\]
This field belongs to $L^{2,\infty}(\mR^2)$ but not to $L^2(\mR^2)$. However, our Theorem \ref{thm:counter} shows that uniform control of the spatial norm alone does not suffice in the time-dependent case: the formula can fail even for drifts in $L^\infty_tL^d_x$.

The time-dependent part of Corollary \ref{cor:classes} applies to the two-dimensional Navier--Stokes vorticity equation. Here $\cM$ denotes finite signed measures with the weak-* topology. Let $\rho\in C([0,T];\cM)\cap C((0,T];L^1\cap L^\infty)$ solve
\[
\partial_t\rho=\Delta\rho-\div(u_\rho\rho),\quad
u_\rho(t)=K*\rho_t, 
\]
where $K(x)=(-x_2,x_1)/(2\pi|x|^2)$ is the Biot--Savart kernel. Let $p^\rho$ be the transition density of the associated decoupled diffusion with generator $\Delta+u_\rho(t)\cdot\nabla$.

\begin{corollary}\label{cor:ns}
Let $p^{\rho}$ be the aforementioned kernel. For every $x,y\in\mR^2$, it holds that 
\[
\lim_{t\downarrow0}t\log p^\rho_{0,t}(x,y)
=-\frac{|x-y|^2}{4}, 
\]
and the convergence is locally uniform in $x,y$.
\end{corollary}

The next theorem compares these heat-kernel conclusions with the law of the full path. It shows that the Euclidean Varadhan formula and strong well-posedness can coexist with a path law singular with respect to Brownian motion.

\begin{theorem}[Singular path law]
\label{thm:path}
Let $d\geq 2$, and let $\mW_0$ be the law of $\sqrt2W$ on $C([0,1];\mR^d)$.  There exists a time-independent, divergence-free drift $u\in L^d(\mR^d;\mR^d)$ such that \eqref{eq:sde}, started from the origin, has a unique strong solution and $\mP\perp\mW_0$.  For the same drift, the transition density is strictly positive, jointly continuous, and satisfies \eqref{eq:var} for every $x,y\in\mR^d$, locally uniformly in $x,y$.
\end{theorem}

\medskip

To sum up, the lower bound has a broad scope because its hypothesis is imposed directly on the antisymmetric form. Theorem \ref{thm:lower} requires only the sector condition \eqref{eq:sector}, allowing an $L^2$-in-time lower-order coefficient and an arbitrary finite sector constant. Its constants depend only on the dimension and these sector data. This implies the locally uniform Euclidean lower bound for every divergence-free drift in $L^\infty_t{\rm BMO}^{-1}_x$, as stated in Corollary \ref{cor:bmo-lower}.

At the weak-$L^d$ endpoint, Corollary \ref{cor:classes} implies the full Euclidean Varadhan formula for this case (time-independent) without any smallness assumption. For time-dependence case, however, the formula can fail even under a uniform $L^d$ bound, with every spatial slice smooth and compactly supported. Theorem \ref{thm:counter} thus shows that control of the critical spatial norm alone does not determine the short-time exponential rate, even within the class covered by the lower bound.

Theorem \ref{thm:path} separates endpoint asymptotics from absolute continuity on path space. Its time-independent $L^d$ drift has a unique strong solution from the origin whose law is mutually singular with that of $\sqrt2W$ on $C([0,1];\mR^d)$. Nevertheless, the same diffusion satisfies the Euclidean Varadhan formula for every pair of endpoints. The Brownian short-time endpoint cost can therefore persist despite singularity of the entire path law.

\begin{figure}[htbp]
\centering
\begingroup
\definecolor{regimelower}{HTML}{154EAA}
\definecolor{regimeupper}{HTML}{A64B00}
\definecolor{regimeboth}{HTML}{075E34}
\definecolor{regimeexample}{HTML}{AE102D}
\definecolor{regimelowerfill}{HTML}{A9C7F5}
\definecolor{regimeupperfill}{HTML}{FFD08A}
\definecolor{regimebothfill}{HTML}{8FD19B}
\scalebox{0.9}{%
\begin{tikzpicture}[
  x=1cm,y=0.9cm,
  every node/.style={font=\small,align=center,inner sep=2pt},
  boundary/.style={line width=1.3pt},
  subclass/.style={line width=1pt,dashed,draw=regimeboth},
  detail/.style={font=\footnotesize,align=center},
  marker/.style={fill=regimeexample,draw=white,line width=0.7pt}
]
% The two ellipses encode sufficient hypotheses; their lens is the overlap.
\fill[regimelowerfill]
  (5.4,4.4) ellipse [x radius=5.2,y radius=3.25];
\fill[regimeupperfill]
  (8,3) ellipse [x radius=5.1,y radius=2.9];
\begin{scope}
  \clip (5.4,4.4) ellipse [x radius=5.2,y radius=3.25];
  \fill[regimebothfill]
    (8,3) ellipse [x radius=5.1,y radius=2.9];
\end{scope}
\draw[boundary,regimelower]
  (5.4,4.4) ellipse [x radius=5.2,y radius=3.25];
\draw[boundary,regimeupper]
  (8,3) ellipse [x radius=5.1,y radius=2.9];

\node[detail,text=black!85] at (6.6,8.02)
  {Divergence-free drifts, $d\geq2$};
\node[text=regimelower] at (4.8,6.98)
  {Condition \eqref{eq:sector}};
\node[detail,text=regimelower] at (4.8,6.37)
  {Euclidean lower bound};

\node[font=\small\bfseries,text=regimeboth] at (7.7,5.36)
  {Full Varadhan formula};
\node[detail,text=regimeboth] at (7.7,4.80)
  {$\displaystyle t\log p_{0,t}(x,y)\to
    -\frac{|x-y|^2}{4}$};

% Both nested ellipses represent time-independent subclasses.
\fill[green!12!white]
  (6.4,3.2) ellipse [x radius=2.32,y radius=1.15];
\draw[subclass]
  (6.4,3.2) ellipse [x radius=2.32,y radius=1.15];
\node[text=regimeboth] at (6.4,3.72) {$L^{d,\infty}$};
\fill[red!25!white]
  (6.4,2.8) ellipse [x radius=1.35,y radius=0.50];
\draw[subclass,draw=regimeexample]
  (6.4,2.8) ellipse [x radius=1.35,y radius=0.50];
\node at (6.4,3.00) {$L^d$};
\filldraw[marker] (5.72,2.63) circle (2.7pt);
\node[detail,anchor=west] at (5.92,2.63) {$\mP\perp\mW_0$};

% A point marks one counterexample in the blue-only part.
\filldraw[marker] (3.25,5.25) circle (3pt);
\node[detail,text width=2.65cm] at (1.97,4.66)
  {$L^\infty_tL^d_x$ example\\[5pt]
   \textcolor{regimeexample}{Upper bound fails}};
\node[text=regimeupper] at (9.35,1.30)
  {Decomposition \eqref{eq:decomp}};
\node[detail,text=regimeupper] at (9.35,0.78)
  {Euclidean upper bound};
\end{tikzpicture}
}
\endgroup
\caption{The blue region represents the lower bound under \eqref{eq:sector}; the orange region represents the matching upper bound under Assumption \ref{ass:decomp}. In their overlap, we obtain the full Varadhan formula for the smooth kernels above and, under the $L^\infty_t {\rm BMO}^{-1}_x$ hypothesis, for the kernels of Corollary \ref{cor:classes}. The dots mark two individual examples: the time-dependent counterexample of Theorem \ref{thm:counter}, and the time-independent $L^d$ drift of Theorem \ref{thm:path} with $\mP\perp\mW_0$.}
\label{fig:regimes}
\end{figure}
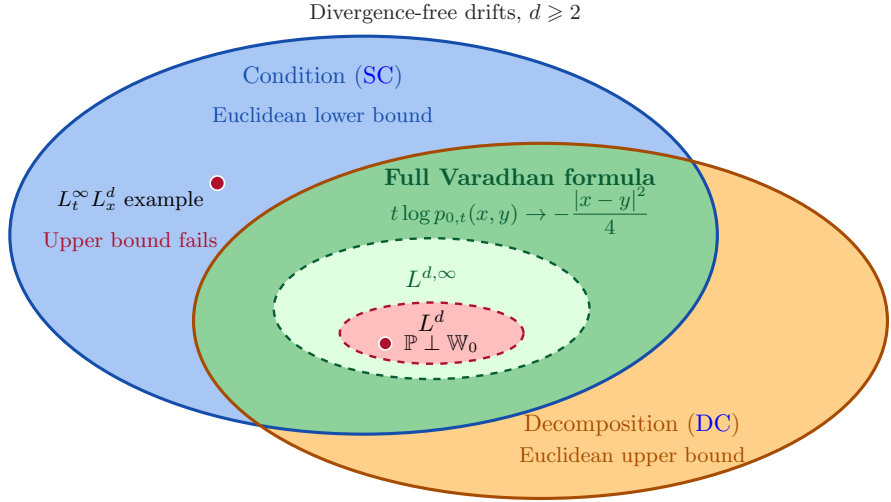

\subsection{Main ideas}
The lower bound uses Gaussian averages of the logarithm of the kernel, following the moving Gaussian approach of Norris \cite{norris1997heat}. The sector condition controls the antisymmetric terms. Entropy dissipation and a first-moment estimate provide the endpoint controls needed for the Gaussian argument. 

The main ingredient in the upper bound is a bounded Lipschitz weight adapted to the spatial distribution of the envelope $F$. Lemma \ref{lem:flatten} constructs an exceptional set on which the gradient of the weight vanishes almost everywhere. A bound on the total radii of its covering balls ensures that the increment of the weight between the endpoints approximates their Euclidean distance with an arbitrarily small additive error. Together with the structural decomposition, this implies a drift estimate with an arbitrarily small gradient coefficient and an integrable lower-order remainder; see Lemma \ref{lem:weights}. A sharp weighted $L^2$ estimate on the middle interval, combined with Nash smoothing on two endpoint intervals of length $t^2$, then retains the Brownian coefficient $1/4$.

Both quantitative estimates for smooth streams are obtained without using the Aronson bounds \eqref{eq:aronson}, but those bounds are used in Appendix \ref{sec:approx} to construct the kernels for rough BMO stream matrices and establish convergence of their smooth approximations.

\subsection{Organization of the paper}
Section \ref{sec:estimates} first proves the quantitative lower bound, then constructs the time-dependent counterexample, and finally proves the upper bound under the structural decomposition. Section \ref{sec:applications} verifies \eqref{eq:sector} and the structural decomposition for concrete classes, proves Corollary \ref{cor:classes}, and presents the Navier--Stokes application. Section \ref{sec:path} proves the path-law result. Appendix \ref{sec:aux} contains the critical embeddings and Lipschitz weights. Appendix \ref{sec:approx} constructs the heat kernel for BMO stream matrices and establishes uniform estimates for its approximations.

\section{Quantitative heat kernel estimates}\label{sec:estimates}
In the proofs of Theorems \ref{thm:lower} and \ref{thm:upper}, $U\in C_b^\infty$ is skew-symmetric, and $p$ is the classical heat kernel of $\Delta+\div U\cdot\nabla$. We keep track of the constants so that the estimates can later be used along smooth approximations. Between these two proofs, Section \ref{sec:counter} presents a bounded-stream counterexample showing why the upper bound needs additional structure.

\subsection{The lower bound}\label{sec:lower}
Throughout this subsection, we put $N:=\|m\|_{L^2(0,T)}$. We use Gaussian averages of the logarithm of the kernel, as in Norris \cite{norris1997heat}. The preliminary kernel upper bound is derived directly from Nash's inequality. The entropy identity then controls the mass outside a ball and implies the endpoint estimates for the moving Gaussian argument, without using \eqref{eq:aronson}.

We first control the drift term generated by a Lipschitz cutoff; this will be used to estimate the first spatial moment of the kernel.

\begin{lemma}\label{lem:cutoff}
Let $A$ be a locally integrable skew matrix whose form satisfies
\[
|\cE_A^a(f,g)|\leq\|\nabla f\|_2
\bigl(M\|\nabla g\|_2+b\|g\|_2\bigr),
\quad f,g\in C_c^\infty(\mR^d),
\]
for constants $M,b\geq0$. Then, for every real bounded Lipschitz $\chi$ and real $\varphi\in H^1(\mR^d)$,
\begin{equation}\label{eq:cutoff}
2|\cE_A^a(\chi\varphi,\varphi)|
\leq C\|\nabla\chi\|_\infty
\bigl(M\|\nabla\varphi\|_2\|\varphi\|_2+b\|\varphi\|_2^2\bigr).
\end{equation}
The form is understood by continuous extension to $H^1(\mR^d)$.
\end{lemma}

\begin{proof}
First take $\varphi\in C_c^\infty$. Fix $s>0$ and choose a smooth quadratic partition of unity $(\theta_Q)_Q$ at scale $s$ such that
\[
\sum_Q\theta_Q^2=1,\quad
\sum_Q|\nabla\theta_Q|^2\leq Cs^{-2},
\]
and each support has diameter at most $Cs$. Put $\varphi_Q=\theta_Q\varphi$ and choose $c_Q$ so that
\begin{equation*}
|\chi-c_Q|\leq Cs\|\nabla\chi\|_\infty
\quad\hbox{on }{\rm supp}\theta_Q.
\end{equation*}
Skew symmetry and $\sum_Q\theta_Q\nabla\theta_Q=0$ imply
\[
\cE_A^a(\chi\varphi,\varphi)
=\sum_Q\cE_A^a((\chi-c_Q)\varphi_Q,\varphi_Q).
\]
The assumed bound yields that
\[
\begin{aligned}
&|\cE_A^a((\chi-c_Q)\varphi_Q,\varphi_Q)|\\
\leq & C\|\nabla\chi\|_\infty
\bigl(s\|\nabla\varphi_Q\|_2+\|\varphi_Q\|_2\bigr)
\bigl(M\|\nabla\varphi_Q\|_2+b\|\varphi_Q\|_2\bigr).
\end{aligned}
\]
By the partition identities and Cauchy--Schwarz, we have
\[
\begin{aligned}
&\sum_Q\|\nabla\varphi_Q\|_2^2 \leq\|\nabla\varphi\|_2^2+Cs^{-2}\|\varphi\|_2^2, \quad
\sum_Q\|\varphi_Q\|_2^2=\|\varphi\|_2^2,\\
&\sum_Q\|\nabla\varphi_Q\|_2\|\varphi_Q\|_2
\leq\|\nabla\varphi\|_2\|\varphi\|_2+Cs^{-1}\|\varphi\|_2^2.
\end{aligned}
\]
Therefore, 
\[
\begin{aligned}
2|\cE_A^a(\chi\varphi,\varphi)|
= & 2\sum_Q\cE_A^a((\chi-c_Q)\varphi_Q,\varphi_Q)\\
\leq & C\|\nabla\chi\|_\infty
\Bigl[M\bigl(s\|\nabla\varphi\|_2^2
+\|\nabla\varphi\|_2\|\varphi\|_2+s^{-1}\|\varphi\|_2^2\bigr)\\
&+b\bigl(s\|\nabla\varphi\|_2\|\varphi\|_2
+\|\varphi\|_2^2\bigr)\Bigr].
\end{aligned}
\]
Taking $s=\|\varphi\|_2/\|\nabla\varphi\|_2$ proves \eqref{eq:cutoff}; if either norm vanishes, then $\varphi=0$. Smooth approximation of $\chi$ on the support of $\varphi$ justifies the identity for Lipschitz $\chi$. Finally, density in $H^1$ proves the stated extension.
\end{proof}

For $r>0$ and $z\in\mR^d$, set
\[
\gamma_{r,z}(\xi)
:=(4\pi r)^{-d/2}
\exp\Big(-\frac{|\xi-z|^2}{4r}\Big).
\]
The next estimate controls the antisymmetric term when the logarithmic heat equation is averaged against this Gaussian.

\begin{lemma}\label{lem:gauss}
Under the sector bound in \eqref{eq:sector}, the map $F\mapsto\cE^a_t(F,\gamma_{r,z})$, initially defined on $C_c^\infty(\mR^d)$, has a unique continuous extension to every $F\in H^1(\gamma_{r,z})$, where this space consists of the functions satisfying
\[
\int\bigl(|F|^2+|\nabla F|^2\bigr)\gamma_{r,z}<\infty.
\]
Moreover, 
\begin{equation}\label{eq:gform}
|\cE^a_t(F,\gamma_{r,z})|
\leq C\bigl(Mr^{-1/2}+m(t)\bigr)
\Big(\int|\nabla F|^2\gamma_{r,z}\Big)^{1/2}.
\end{equation}
\end{lemma}

\begin{proof}
Put $\gamma:=\gamma_{r,z}$. For $F\in C_c^\infty$, write $\bar F:=F-\int F\gamma$. 
By the product rule, we have
\[
\nabla\gamma^{1/4}=\tfrac14\gamma^{1/4}\nabla\log\gamma,
\quad
\nabla(\gamma^{3/4}\bar F)
=\gamma^{3/4}\nabla F+\tfrac34\gamma^{3/4}\bar F\nabla\log\gamma.
\]
Since $U\nabla\log\gamma\cdot\nabla\log\gamma=0$, we have
\[
\cE_t^a(\gamma^{1/4},\gamma^{3/4}\bar F)
=\tfrac14\cE_t^a(\gamma,F)
=-\tfrac14\cE_t^a(F,\gamma).
\]
Thus, \eqref{eq:sector} implies
\begin{align}\label{eq:GFE1}
|\cE_t^a(F,\gamma)|
\leq4\|\nabla\gamma^{1/4}\|_2
\bigl(M\|\nabla(\gamma^{3/4}\bar F)\|_2
+m(t)\|\gamma^{3/4}\bar F\|_2\bigr).
\end{align}

The identity $\nabla\log\gamma(\xi)=-(\xi-z)/(2r)$ and the change of variables $w=(\xi-z)/\sqrt r$ show that
\[
\|\nabla\gamma^{1/4}\|_2^2\leq Cr^{d/4-1},\quad
\|\gamma^{1/2}\|_\infty\leq Cr^{-d/4},\quad
\bigl\||\nabla\log\gamma|^2\gamma^{1/2}\bigr\|_\infty
\leq Cr^{-d/4-1}.
\]
By Gaussian Poincar\'e inequality, we have 
\begin{equation}\label{eq:gaussian-poincare}
\int|\bar F|^2\gamma\leq2r\int|\nabla F|^2\gamma. 
\end{equation}
Using the product rule above, we obtain
\[
\begin{aligned}
\|\nabla(\gamma^{3/4}\bar F)\|_2^2
&\leq2\int\gamma^{3/2}|\nabla F|^2
+\frac98\int\gamma^{3/2}|\bar F|^2|\nabla\log\gamma|^2
\leq Cr^{-d/4}\int|\nabla F|^2\gamma,\\
\|\gamma^{3/4}\bar F\|_2^2
&\leq Cr^{-d/4}\int|\bar F|^2\gamma
\leq Cr^{1-d/4}\int|\nabla F|^2\gamma.
\end{aligned}
\]
This together with \eqref{eq:GFE1} proves \eqref{eq:gform}. Finally, $C_c^\infty$ is dense in $H^1(\gamma)$, and spatial cutoffs converge in this norm. Hence the estimate extends to $H^1(\gamma)$, and the extension agrees with the stated cutoff limit.
\end{proof}

The next moment bound keeps a fixed fraction of the kernel mass within a ball of radius proportional to $\sqrt h$, providing the input for the logarithmic estimates. 

\begin{lemma}\label{lem:mass}
Under the assumptions of Theorem \ref{thm:lower}, there is $C$ such that, for $0\leq s<s+h\leq T$,
\begin{equation}\label{eq:mass}
\int|z-x|p_{s,s+h}(x,z)\,\d z\leq C\sqrt h.
\end{equation}
The same estimate holds for the time-reversed adjoint kernel. Here $C$ depends only on $d,M,N$.
\end{lemma}

\begin{proof}
Fix $s,x$ and write $k_t(z):=p_{s,s+t}(x,z)$. Since the drift is divergence-free, the forward equation and its adjoint preserve mass, and their $L^2$ energy identities contain only the Dirichlet energy. The classic Nash inequality therefore implies
\begin{equation}\label{eq:nash}
\|k_t\|_\infty\leq Ct^{-d/2},
\end{equation}
see for instance \cite{carlen1987upper}. 

Let $M(t):=\int|z-x|k_t(z)\,\d z=W_1(k_t, \delta_x)$. To estimate $M(t)$, let $\chi$ be a bounded smooth $1$-Lipschitz approximation to $|z-x|$, with $\chi(x)=0$. Testing the forward equation $\partial_t k_t=\div((I-U(s+t))\nabla k_t)$ against $\chi$, and using \eqref{eq:cutoff}, we obtain 
\begin{equation*}
\begin{aligned}
\frac{\d}{\d t}\int\chi k_t
&=-\int(I-U(s+t))\nabla k_t\cdot\nabla\chi\\
&=-2\int\sqrt{k_t}\,\nabla\sqrt{k_t}\cdot\nabla\chi
-2\cE_{s+t}^a(\chi\sqrt{k_t},\sqrt{k_t})\\
&\leq C \Big( \int|\nabla\sqrt{k_t}|^2\Big)^{\frac{1}{2}}+Cm(s+t),
\end{aligned}
\end{equation*}
where we used the fact that 
\[
2\cE_{s+t}^a(\chi\sqrt{k_t},\sqrt{k_t})
=2\int\sqrt{k_t}\,U(s+t)\nabla\chi\cdot\nabla\sqrt{k_t}
=\int U(s+t)\nabla\chi\cdot\nabla k_t.
\]
Define the entropy and Fisher information as 
\[
H(t):=\int k_t\log k_t, \quad I(t):= \frac{1}{4}\int \frac{|\nabla k_t|^2}{k_t}=\int|\nabla\sqrt{k_t}|^2.
\]
Letting $\chi\uparrow|z-x|$, we get 
\begin{align}\label{eq:d_moment}
  M'(t)\leq C\sqrt{I(t)}+Cm(s+t).
\end{align}
The entropy dissipation identity is
\begin{equation}\label{eq:entropy}
H'(t)=-\int \nabla k_t \cdot \nabla (\log k_t) = -4I(t).
\end{equation}

Below we want to get the moment estimate via entropy. First, Nash's bound \eqref{eq:nash} yields 
\begin{align}\label{eq:H_upper}
  H(t)+(d/2)\log t\leq C. 
\end{align}
For the lower bound of $H(t)$, use the probability density $q_t(z):=Z_d^{-1}t^{-d/2}\exp(-|z-x|/\sqrt t)$, where $Z_d:=\int\exp(-|z|)\,\d z$. By Jensen's inequality, we obtain
\begin{align}\label{eq:H_lower}
  H(t)=\int k_t\log q_t - \int k_t\log \frac{q_t}{k_t} \geq\int k_t\log q_t
=-\log Z_d-\frac d2\log t-\frac{M(t)}{\sqrt t}.
\end{align}
Put
\[
F(t):=2M(t)+\sqrt t\Bigl(H(t)+\frac d2\log t+\log Z_d\Bigr).
\]
Then the two entropy bounds \eqref{eq:H_upper} and \eqref{eq:H_lower} imply
\[
M(t)\leq F(t)\leq2M(t)+C\sqrt t.
\]
By \eqref{eq:d_moment}, \eqref{eq:entropy} and \eqref{eq:H_upper}, we get 
\[
\begin{aligned}
F'(t)
&=2M'(t)-4\sqrt t\,I(t)
+\frac{H(t)+(d/2)\log t+\log Z_d+d}{2\sqrt t}\\
&\leq C\sqrt{I(t)}-4\sqrt t\,I(t)
+\frac{C}{\sqrt t}+Cm(s+t)\leq\frac{C}{\sqrt t}+Cm(s+t).
\end{aligned}
\]
Therefore, 
\[
\begin{aligned}
M(h)\leq F(h)\leq C\sqrt h+C\int_0^h m(s+t)\,\d t\leq C\sqrt h,
\end{aligned}
\]
where $C$ depends only on $d,M,N$. This proves \eqref{eq:mass}.

Reversing time and replacing $U$ by $-U$ leaves the hypotheses unchanged, so the same proof applies to the adjoint kernel.
\end{proof}

The moving Gaussian argument needs initial and terminal bounds for the averaged logarithm of the kernel; the next lemma provides these from the preceding estimates.

\begin{lemma}\label{lem:endpoints}
Under the assumptions of Theorem \ref{thm:lower}, there is $C=C(d,M,N)$ such that, for $0\leq s<s+h\leq T$ and $x,y\in\mR^d$,
\begin{equation}\label{eq:anchor}
\begin{aligned}
\int\log p_{s,s+h/8}(x,z)\gamma_{h,x}(z)\,\d z
&\geq-\frac d2\log h-C,\\
\int\log p_{s+7h/8,s+h}(z,y)\gamma_{h,y}(z)\,\d z
&\geq-\frac d2\log h-C.
\end{aligned}
\end{equation}
\end{lemma}

\begin{proof}
We only need to prove the first estimate in \eqref{eq:anchor}. Rescale the kernel to unit time and move its starting point to the origin by setting
\[
k_\theta(z):=h^{d/2}p_{s,s+h\theta}(x,x+\sqrt h\,z),
\quad 0<\theta\leq1.
\]
This is a probability density with stream matrix $\widetilde U(\theta,z):=U(s+h\theta,x+\sqrt h\,z)$. The rescaled form satisfies \eqref{eq:sector} with the same $M$ and with $\widetilde m(\theta):=\sqrt h\,m(s+h\theta)$. In particular,
\[
\int_0^1\widetilde m(\theta)^2\,\d\theta
=\int_s^{s+h}m(t)^2\,\d t\leq N^2.
\]
Write $\gamma:=\gamma_{1,0}$. Changing variables in the Gaussian average, we obtain
\begin{equation}\label{eq:endpoint-rescaling}
\int\log p_{s,s+h/8}(x,z)\gamma_{h,x}(z)\,\d z
=-\frac d2\log h+\int\log k_{1/8}(z)\gamma(z)\,\d z.
\end{equation}
Thus, it remains to bound the last integral from above by a constant $C$ depending only on $d,M,N$.

We work on $1/16\leq\theta\leq1/8$, away from the initial singularity. On this interval, by \eqref{eq:nash} and \eqref{eq:mass}, respectively, we have
\[
0\leq k_\theta\leq A,\qquad \int|z|k_\theta(z)\,\d z\leq C,
\]
where $A\geq1$ and $C$ depend only on $d,M,N$. Choose $R:=\max\{1,4C\}$. Since $k_\theta$ has total mass one,
\[
\int_{B_R(0)^c}k_\theta\leq\frac1R\int|z|k_\theta(z)\,\d z\leq\frac14,
\qquad \int_{B_R(0)}k_\theta\geq\frac34.
\]
Set $b:=(4|B_R(0)|)^{-1}$ and $E_\theta:=\{z\in B_R(0):k_\theta(z)\geq b\}$. Then
\[
A|E_\theta|\geq\int_{E_\theta}k_\theta
\geq\frac34-b|B_R(0)|=\frac12, \quad 
\int_{E_\theta}\gamma 
\geq\mu:=\frac{(4\pi)^{-d/2}\e^{-R^2/4}}{2A}>0.
\]
Define 
\[
G(\theta)=\int \log k_{\theta} \, \gamma,\quad
I(\theta)=\int|\nabla \log k_{\theta}|^2\gamma.
\]
The forward equation $\partial_\theta k_\theta=\div((I-\widetilde U)\nabla k_\theta)$ and skew symmetry imply
\[
\partial_\theta \log k_{\theta}
=\div((I-\widetilde U)\nabla \log k_{\theta})+|\nabla \log k_{\theta}|^2.
\]
Thus, integrating by parts, we obtain
\[
G'(\theta)=I(\theta)-\int\nabla \log k_{\theta}\cdot\nabla\gamma
+\cE^a_{\widetilde U(\theta)}(\log k_{\theta},\gamma).
\]
Since $\nabla\log\gamma=-z/2$ and $\int|\nabla\log\gamma|^2\gamma=d/2$, Cauchy--Schwarz and \eqref{eq:gform} yield
\begin{equation}\label{eq:endpoint-log-differential}
\begin{aligned}
G'(\theta)\geq I(\theta)-\bigl[\sqrt{d/2}+C(M+\widetilde m(\theta))\bigr]\sqrt{I(\theta)} \geq\tfrac12I(\theta)-F(\theta),
\end{aligned}
\end{equation}
where Young's inequality is used in the last step, and
\[
F(\theta):=C\bigl(1+\widetilde m(\theta)^2\bigr),\qquad
\int_{1/16}^{1/8}F(\theta)\,\d\theta\leq C.
\]
Let $L:=\max\{0,-\log b\}$. Since $\log k_\theta\geq-L$ on $E_\theta$ and $\int_{E_\theta}\gamma\geq\mu$, we obtain from \eqref{eq:gaussian-poincare}
\[
2I(\theta)\geq\int|\log k_\theta-G(\theta)|^2\gamma
\geq\mu(-L-G(\theta))_+^2.
\]
Absorb the accumulated error by setting
\[
q(\theta):=\Big(-L-G(\theta)-\int_{1/16}^\theta F(v)\,\d v\Big)_+.
\]
On $\{q>0\}$, we have $q\leq-L-G$, so \eqref{eq:endpoint-log-differential} yields
\[
q'=-G'-F\leq-\tfrac12I\leq-\frac\mu4q^2.
\]
The inequality $q'\leq-\mu q^2/4$ also holds almost everywhere on $\{q=0\}$. Hence $q$ is nonincreasing; integrating $(1/q)'\geq\mu/4$ on intervals where $q>0$, we obtain
\[
q(\theta)\leq\frac4{\mu(\theta-1/16)},\qquad 1/16<\theta\leq1/8.
\]
At $\theta=1/8$, we conclude that
\[
\int\log k_{1/8}(z)\gamma(z)\,\d z
=G(1/8)\geq-L-\frac{64}{\mu}
-\int_{1/16}^{1/8}F(v)\,\d v\geq-C, 
\]
where $C$ depend only on $d,M,N$. So we complete 
\end{proof}

\begin{proof}[Proof of Theorem \ref{thm:lower}]
Fix $0<t<T\wedge \frac{1}{4}$, $x,y\in\mR^d$, and $0<\beta<1/4$. Put
\[
a_\beta:=\frac{\beta t}{2},\quad
b_\beta:=\Big(1-\frac\beta2\Big)t,\quad
I_\beta:=[a_\beta,b_\beta]. 
\]
On $I_\beta$, let the Gaussian center move at constant velocity:
\[
\alpha_\tau:=x+\frac{\tau-a_\beta}{(1-\beta)t}(y-x),\quad
|\alpha'_\tau|=\frac{|x-y|}{(1-\beta)t}.
\]
Choose a varying heat-time parameter
\[
r(\tau):=\min\{8\tau,t,8(t-\tau)\},\quad
\Psi_\tau:=\gamma_{r(\tau),\alpha_\tau}.
\]
Thus, $r(a_\beta)=r(b_\beta)=4\beta t$ and $r(t/2)=t$. The function $r$ is positive and Lipschitz on $I_\beta$, with $|r'|\leq8$. $\Psi_\tau$ satisfies 
\[
\partial_\tau\Psi_\tau
=r'(\tau)\Delta\Psi_\tau-\alpha'_\tau\cdot\nabla\Psi_\tau.
\]

Define
\[
F_\tau(z):=p_{0,\tau}(x,z),\quad
G_\tau(z):=p_{\tau,t}(z,y),
\]
and set
\[
f(\tau):=\int\log F_\tau\,\Psi_\tau,\quad
g(\tau):=\int\log G_\tau\,\Psi_\tau.
\]
Since $u$ is divergence-free, the forward and backward equations are
\[
\partial_\tau F_\tau=\Delta F_\tau-u(\tau)\cdot\nabla F_\tau,
\quad
-\partial_\tau G_\tau=\Delta G_\tau+u(\tau)\cdot\nabla G_\tau.
\]

Put
\[
A_F(\tau)^2:=\int|\nabla\log F_\tau|^2\Psi_\tau,\quad
A_G(\tau)^2:=\int|\nabla\log G_\tau|^2\Psi_\tau.
\]
Differentiating $f$ and $g$, integrating by parts and using the fact that $-\int u(\tau)\cdot\nabla H\,\Psi_\tau=\cE_\tau^a(H,\Psi_\tau)$, we obtain
\begin{equation*}
\begin{aligned}
f'(\tau)
=&A_F(\tau)^2+\cE_\tau^a(\log F_\tau,\Psi_\tau)+\int\bigl[\alpha'_\tau-(1+r'(\tau))\nabla\log\Psi_\tau\bigr]
 \cdot\nabla\log F_\tau\,\Psi_\tau,\\
-g'(\tau)
=&A_G(\tau)^2-\cE_\tau^a(\log G_\tau,\Psi_\tau)-\int\bigl[\alpha'_\tau+(1-r'(\tau))\nabla\log\Psi_\tau\bigr]
 \cdot\nabla\log G_\tau\,\Psi_\tau.
\end{aligned}
\end{equation*}
By Cauchy--Schwarz and Lemma \ref{lem:gauss}, using the fact that $\int|\nabla\log\Psi_\tau|^2\Psi_\tau=d/(2r(\tau))$, one sees that both right-hand sides are bounded below by $A^2-(|\alpha'_\tau|+B(\tau))A$, with
\[
B(\tau):=C\frac{1+|r'(\tau)|}{\sqrt{r(\tau)}}+Cm(\tau).
\]
Completing the square, we obtain
\begin{equation}\label{eq:Dfg}
f'(\tau),\ -g'(\tau)
\geq-\frac14\bigl(|\alpha'_\tau|+B(\tau)\bigr)^2.
\end{equation}
We will integrate this square before estimating its mixed term.

Apply \eqref{eq:anchor} with $h:=4\beta t$, and $s=0$ or $t-h$. Since $a_\beta=t-b_\beta=h/8$, we obtain
\begin{equation}\label{eq:fg_end}
f(a_\beta),\ g(b_\beta)
\geq-\frac d2\log(4\beta t)-C
\geq-\frac d2\log t-C.
\end{equation}
Integrate the first inequality in \eqref{eq:Dfg} from $a_\beta$ to $t/2$, and the second from $t/2$ to $b_\beta$. \eqref{eq:fg_end} yields
\[
f(t/2)+g(t/2)\geq-d\log t-C-\frac{|x-y|^2}{4t(1-\beta)}
-\frac{|x-y|J_1}{2t(1-\beta)}-\frac{J_2}{4},
\]
where 
\[
J_1:=\int_{I_\beta}B(\tau)\,\d\tau,\quad
J_2:=\int_{I_\beta}B(\tau)^2\,\d\tau.
\]
By the definition of $r$, 
\[
\int_{I_\beta}\frac{\d\tau}{\sqrt{r(\tau)}}\leq\frac54\sqrt t,
\quad
\int_{I_\beta}\frac{\d\tau}{r(\tau)}
=\frac34+\frac14\log\frac1{4\beta}.
\]
Since $|r'|\leq8$, $\int_{I_\beta}m\leq\sqrt t\,N$, and $\int_{I_\beta}m^2\leq N^2$, it follows that
\[
J_1\leq C\sqrt t,\quad
J_2\leq C\bigl(1+|\log\beta|\bigr).
\]
Thus, for any $0<\varepsilon\ll 1$
\[
\frac{|x-y|J_1}{2t(1-\beta)}\leq\frac{\varepsilon |x-y|^2}{4t(1-\beta)(1-\varepsilon)}
+\frac{(1-\varepsilon)J_1^2}{4\varepsilon t(1-\beta)} \leq\frac{\varepsilon |x-y|^2}{4t(1-\beta)(1-\varepsilon)}
+\frac{C}\varepsilon.
\]
Consequently,
\begin{equation}\label{eq:fg_mid}
\begin{aligned}
f(t/2)+g(t/2)\geq -d\log t
-\frac{|x-y|^2}{4t(1-\beta)(1-\varepsilon)}-C\bigl(\varepsilon^{-1}+1+|\log\beta|\bigr).
\end{aligned}
\end{equation}

At the midpoint, $\Psi_{t/2}=\gamma_{t,\alpha_{t/2}}$ is a probability density. Using Chapman--Kolmogorov equation and Jensen's inequality, we obtain
\begin{equation}\label{eq:logp}
\begin{aligned}
\log p_{0,t}(x,y)
&=\log\int\frac{F_{t/2}G_{t/2}}{\Psi_{t/2}}\,\Psi_{t/2} \geq f(t/2)+g(t/2)-\int\Psi_{t/2}\log\Psi_{t/2} \\
&=f(t/2)+g(t/2)
+\frac d2\log(4\pi t)+\frac d2.
\end{aligned}
\end{equation}
Here we used the fact that $\log\Psi_{t/2}(z)=-(d/2)\log(4\pi t)-|z-\alpha_{t/2}|^2/(4t)$ and $\int|z-\alpha_{t/2}|^2\Psi_{t/2}(z)\,\d z=2dt$.

Choose $\beta:=\sqrt t/2$ and $\varepsilon:=\sqrt t$. Since $t<1/4$, we have $0<\beta<1/4$ and $0<\varepsilon<1/2$. Moreover,
\[
\begin{aligned}
(1-\beta)(1-\varepsilon)
&=(1-\sqrt t/2)(1-\sqrt t)\geq(1-\sqrt t)^2,\\
\varepsilon^{-1}+1+|\log\beta|
&=t^{-1/2}+1+\log(2/\sqrt t)\leq\frac C{\sqrt t}.
\end{aligned}
\]
Substituting these bounds into \eqref{eq:fg_mid} and \eqref{eq:logp}, and combining $-d\log t$ with $(d/2)\log(4\pi t)$, we obtain
\[
\log p_{0,t}(x,y)\geq-\frac d2\log t
-\frac{|x-y|^2}{4t(1-\sqrt t)^2}-\frac C{\sqrt t}, 
\]
where $C$ depends only on $d,M,N$. So we complete our proof.
\end{proof}

\begin{proof}[Proof of Corollary \ref{cor:bmo-lower}]
The smooth approximations in Appendix \ref{sec:approx} satisfy \eqref{eq:sector} with $m=0$ and a common $M$ by Proposition \ref{prop:bounds}(i). By Theorem \ref{thm:lower}, we therefore obtain \eqref{eq:lower} for their kernels with one constant $C$, depending only on $d,M$. By Theorem \ref{thm:feller}(b), this estimate passes to $p$ at each fixed $0<t<1/4$. Taking logarithms and multiplying by $t$ yields
\[
t\log p_{0,t}(x,y)
\geq-\frac d2t\log t
-\frac{|x-y|^2}{4(1-\sqrt t)^2}-C\sqrt t.
\]
The right-hand side converges to $-|x-y|^2/4$ uniformly for $x,y$ in any fixed compact set, proving the claim.
\end{proof}

\subsection{A time-dependent counterexample}\label{sec:counter}

The lower bound above holds for every bounded skew stream. We now construct one for which the Euclidean upper bound fails. On shrinking time intervals, a moving ball carries the diffusion a fixed distance; in the moving coordinates the process is Brownian motion with generator $\Delta$ until exit. The stream is bounded but need not be smooth in time, and we use the heat kernel constructed in Appendix \ref{sec:approx}.

\begin{proof}[Proof of Theorem \ref{thm:counter}]
Choose $\chi\in C_c^\infty(B_3(0))$ with $\chi=1$ on $B_2(0)$, and define the skew-symmetric matrix $Q$ on $\mR^d$ by
\[
Q_{21}(z):=z_2\chi(z),\quad
Q_{12}(z):=-z_2\chi(z),
\]
with all other entries equal to zero.  Set $q:=\div Q$.  Then
\[
\begin{aligned}
q_1&=\partial_2(z_2\chi)=\chi+z_2\partial_2\chi,\\
q_2&=\partial_1(-z_2\chi)=-z_2\partial_1\chi,
\quad q_j=0\quad(j\geq3).
\end{aligned}
\]
Consequently,
\[
q=e_1\quad\hbox{on }B_2(0),
\quad q\in C_c^\infty(B_3(0);\mR^d).
\]
Moreover, the skew symmetry of $Q$ implies $\div q=\sum_{i,j}\partial_j\partial_iQ_{ij}=0$ in the sense of distributions.

Let $\lambda_1$ be the first Dirichlet eigenvalue of $-\Delta$ on $B_1(0)$ and choose
\[
M=3\sqrt{\lambda_1}.
\]
Starting with a sufficiently small $t_1\in(0,1)$, set
\[
t_{n+1}:=t_n^3,
\quad I_n:=(t_{n+1},t_n],
\quad \delta_n:=t_n-t_{n+1},
\quad r_n:=M\delta_n.
\]
For $t\in[t_{n+1},t_n]$, put
\[
\gamma_n(t):=\frac{t-t_{n+1}}{\delta_n}e_1.
\]
For each $(t,x)\in \mR_+\times \mR^d$, define
\begin{equation}\label{eq:plug}
\begin{aligned}
U(t,x)&:=M \sum_{n=1}^\infty Q\Big(\frac{x-\gamma_n(t)}{r_n}\Big)\1_{I_n}(t),\\
u(t,x)&:=\div U(t,x)=\sum_{n=1}^\infty \delta_n^{-1}q\Big(\frac{x-\gamma_n(t)}{r_n}\Big)\1_{I_n}(t).
\end{aligned}
\end{equation}
For $t\in I_n$, its spatial slice is smooth and supported in $B_{3r_n}(\gamma_n(t))$.  Moreover,
\[
\begin{aligned}
\|u(t)\|_{L^d}^d=\Big(\frac{M}{r_n}\Big)^d
\int
\left|q\Big(\frac{x-\gamma_n(t)}{r_n}\Big)\right|^d\,\d x=M^d\|q\|_{L^d}^d.
\end{aligned}
\]
Thus, $u\in L^\infty_tL^d_x$, and every spatial slice is smooth and compactly supported.  

Moreover, $\|U\|_\infty\leq M\|Q\|_\infty$, so Theorem \ref{thm:feller} ensures the existence of the heat kernel and Markov family used below. On each interval $I_n$, the coefficients admit a smooth bounded extension in time. The classical evolution on that interval agrees with the evolution constructed there by uniqueness of the backward energy equation; the values at the block endpoints do not affect the evolution.

According to \eqref{eq:plug}, on the moving ball $B_{2r_n}(\gamma_n(t))$, we have
\[
u(t,x)=\delta_n^{-1}e_1
=\dot\gamma_n(t), \quad x \in B_{2r_n}(\gamma_n(t)).
\]
The two relevant scale relations are
\begin{equation}\label{eq:scales}
\frac{\delta_n}{r_n^2}
=\frac{1}{M^2\delta_n}\to\infty ~~\mbox{ and }~~
\frac{r_n^2}{t_{n+1}}
=M^2\frac{(1-t_n^2)^2}{t_n}
\to\infty, ~~\mbox{ as } n\to\infty.
\end{equation}
The first determines the confinement cost on the last block; the second makes the earlier time interval negligible on the scale $r_n$. Figure \ref{fig:moving-ball} illustrates the two time intervals.

\begin{figure}[htbp]
\centering
\scalebox{0.9}{%
\begin{tikzpicture}[
  x=1cm,
  y=1cm,
  >=latex,
  every node/.style={font=\small},
  interval/.style={font=\scriptsize,fill=white,inner sep=1.5pt},
  endpoint/.style={font=\scriptsize,align=left}
]
% Coordinates: time is horizontal and only the transport coordinate x_1 is shown.
\def\tzero{0.8}
\def\tb{4.3}
\def\tc{10.2}
\def\yzero{0.8}
\def\yone{4.0}

% Coordinate axes and endpoint levels.
\draw[->,line width=0.7pt] (0.55,0.35)--(10.75,0.35)
  node[right] {$t$};
\draw[->,line width=0.7pt] (\tzero,0.35)--(\tzero,4.75)
  node[above] {$x_1$};
\draw[densely dotted,black!35] (\tzero,\yzero)--(10.45,\yzero);
\draw[densely dotted,black!35] (\tzero,\yone)--(10.45,\yone);
\node[left,font=\scriptsize] at (\tzero,\yzero) {$0$};
\node[left,font=\scriptsize] at (\tzero,\yone) {$1$};

% The projected moving core on the last block I_n.
\path[fill=green!16]
  (\tb,0.44)--(\tc,3.64)--(\tc,4.36)--(\tb,1.16)--cycle;
\draw[green!65,line width=0.55pt]
  (\tb,0.44)--(\tc,3.64)
  (\tb,1.16)--(\tc,4.36);
\draw[orange!80!black,dashed,line width=0.7pt]
  (\tb,\yzero)--(\tc,\yone);

% Earlier evolution from (0,0) to the start of the last block.
\draw[->,blue,line width=1.15pt]
  plot[smooth] coordinates {
    (0.80,0.80) (1.15,1.12) (1.55,0.92) (1.95,1.30)
    (2.35,0.95) (2.80,1.18) (3.25,0.88) (3.70,1.02)
    (4.30,0.80)
  };
% The continuation inside the moving ball on I_n.
\draw[->,red,line width=1.15pt]
  plot[smooth] coordinates {
    (4.30,0.80) (4.90,1.05) (5.50,1.55) (6.10,1.72)
    (6.70,2.20) (7.30,2.35) (7.90,2.84) (8.50,3.00)
    (9.10,3.52) (9.70,3.64) (10.20,4.02)
  };
\path (\tb,\yzero)--(\tc,\yone)
  node[pos=0.57,below=1.5pt,sloped,font=\scriptsize,
    inner sep=0.4pt,text=red] {$\gamma_n(t)$}
  node[pos=0.84,below=1.5pt,sloped,font=\scriptsize,
    inner sep=0.4pt,text=red] {$X_{t_n}\approx e_1$};

% Time labels and the endpoint events.
\foreach \x in {\tzero,\tb,\tc}{
  \draw (\x,0.29)--(\x,0.41);
}
\node[below,font=\scriptsize] at (\tzero,0.29) {$0$};
\node[below,font=\scriptsize] at (\tb,0.29) {$t_{n+1}$};
\node[below,font=\scriptsize] at (\tc,0.29) {$t_n$};
\draw[<->,black!60] (\tzero,4.53)
  --node[interval] {$[0,t_{n+1}]$}(\tb,4.53);
\draw[<->,black!60] (\tb,4.53)
  --node[interval] {$I_n=(t_{n+1},t_n]$}(\tc,4.53);
\fill[blue!80!black] (\tzero,\yzero) circle (1.7pt);
\fill[blue!80!black] (\tb,\yzero) circle (1.7pt);
\fill[red!85!black] (\tc,\yone) circle (1.7pt);
\node[endpoint,right,text=blue!80!black] at (\tzero+0.08,0.58)
  {$X_0=0$};
\node[endpoint,left,text=blue!80!black] at (\tb-0.08,1.42)
  {$X_{t_{n+1}}\in B_{\frac{r_n}{2}}(0)$};
\end{tikzpicture}
}
\caption{The blue part starts from $X_0=0$ and ends inside $B_{\frac{r_n}{2}}(0)$ at time $t_{n+1}$ with probability close to 1.  During $I_n=(t_{n+1},t_n]$, the red continuation depicts a path that stays in the moving ball $B_{r_n}(\gamma_n(t))$ and reaches a neighborhood of $e_1$ at $t_n$.}
\label{fig:moving-ball}
\end{figure}
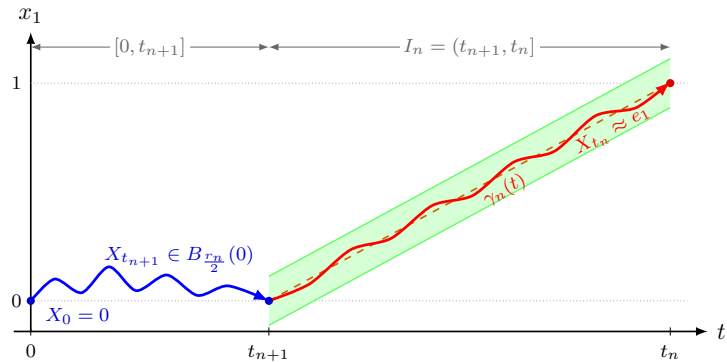

For each $z\in B_{r_n}(0)$, work under $\mP_{t_{n+1},z}$, so that $X_{t_{n+1}}=z$ almost surely, and set
\[
\sigma_n
:=\inf\{t: t \in [t_{n+1},t_n], ~ X_t\notin B_{r_n}(\gamma_n(t))\}.
\]
We use the convention $\inf\varnothing:=+\infty$; hence $\{\sigma_n>t_n\}$ is exactly the event that the diffusion remains in the moving ball throughout the block $I_n$.

Before $\sigma_n$, the relative process $Y_t:=X_t-\gamma_n(t)$ satisfies
\[
\d Y_t
=\bigl(u(t,X_t)-\dot\gamma_n(t)\bigr)\,\d t
 +\sqrt2\,\d W_t
=\sqrt2\,\d W_t.
\]
So $Y$ is Brownian motion with generator $\Delta$, killed on leaving $B_{r_n}(0)$.  Let $k_s^{B_r(0)}$ denotes the Dirichlet heat kernel of $\Delta$ in $B_r(0)$, then, for every $z\in B_{r_n}(0)$ and every $y\in B_{r_n}(e_1)$, it holds that
\[
\mP_{t_{n+1},z}
\bigl(X_{t_n}\in\d y,\ \sigma_n>t_n\bigr)
=k_{\delta_n}^{B_{r_n}(0)}(z,y-e_1)\,\d y.
\]
By domination of the killed transition measure and continuity of both kernels, for each $z\in B_{\frac{r_n}{2}}(0)$,
\[
p_{t_{n+1},t_n}(z,e_1)
\geq k_{\delta_n}^{B_{r_n}(0)}(z,0)
=r_n^{-d}
k^{B_1(0)}_{\delta_n/r_n^2}(z/r_n,0).
\]

Let $\phi_1$ be the positive first Dirichlet eigenfunction of $-\Delta$ on $B_1(0)$.  By the eigenfunction expansion, we have
\[
\e^{\lambda_1t}k_t^{B_1(0)}(x,0)
\to\phi_1(x)\phi_1(0), \quad t\to\infty, 
\]
uniformly for $x\in B_{1/2}(0)$.  Since $\phi_1$ has a positive lower bound on $B_{1/2}(0)$, \eqref{eq:scales} implies, for all sufficiently large $n$ and all $z\in B_{\frac{r_n}{2}}(0)$,
\begin{align}
p_{t_{n+1},t_n}(z,e_1)
&\geq C^{-1}r_n^{-d}
\exp\Big\{-\lambda_1\frac{\delta_n}{r_n^2}\Big\}
=C^{-1}r_n^{-d}
\exp\Big\{-\frac{\lambda_1}{M^2\delta_n}\Big\}.
\label{eq:plug-lower}
\end{align}

We next show that, before the last block begins, almost all mass is already inside $B_{\frac{r_n}{2}}(0)$.  In fact, the rough Gaussian upper bound \eqref{eq:aronson} implies
\[
\begin{aligned}
\int_{|z|\geq r_n/2}p_{0,t_{n+1}}(0,z)\,\d z
&\leq
Ct_{n+1}^{-d/2}
\int_{|z|\geq r_n/2}
\exp\Big\{-\frac{|z|^2}{Ct_{n+1}}\Big\}\,\d z\\
&\leq
C\exp\Big\{-\frac{r_n^2}{Ct_{n+1}}\Big\}
=o(1),
\end{aligned}
\]
where the last step follows from \eqref{eq:scales}.  By conservativity,
\begin{equation}\label{eq:plug-mass}
\int_{B_{\frac{r_n}{2}}(0)}p_{0,t_{n+1}}(0,z)\,\d z=1-o(1),
\end{equation}
as claimed.

Finally, Chapman--Kolmogorov equation yields that
\[
\begin{aligned}
p_{0,t_n}(0,e_1)
&=\int
p_{0,t_{n+1}}(0,z)
p_{t_{n+1},t_n}(z,e_1)\,\d z\\
&\geq
\int_{B_{\frac{r_n}{2}}(0)}
p_{0,t_{n+1}}(0,z)
p_{t_{n+1},t_n}(z,e_1)\,\d z.
\end{aligned}
\]
Combining \eqref{eq:plug-lower} and \eqref{eq:plug-mass}, we get
\[
p_{0,t_n}(0,e_1)
\geq (1-o(1))C^{-1}r_n^{-d}
\exp\Big\{-\frac{\lambda_1}{M^2\delta_n}\Big\},
\]
which implies
\[
\begin{aligned}
t_n\log p_{0,t_n}(0,e_1)
\geq{}&
-\frac{\lambda_1}{M^2}\frac{t_n}{\delta_n}
-dt_n\log r_n\\
&-t_n\log C+t_n\log(1-o(1)).
\end{aligned}
\]
By definition of $\delta_n$ and $r_n$, we have
\[
\frac{t_n}{\delta_n}=\frac1{1-t_n^2}\to1,
\quad
r_n=Mt_n(1-t_n^2),
\quad
t_n|\log r_n|\to0.
\]
Noting that $M^2=9\lambda_1$, one sees that
\[
\liminf_{n\to\infty}
t_n\log p_{0,t_n}(0,e_1)
\geq-\frac{\lambda_1}{M^2}
=-\frac{1}{9}, 
\]
which proves the failure of the Euclidean upper bound. Here $C$ depends only on $d$ and the fixed cutoff $\chi$, and is independent of $n$.
\end{proof}

\subsection{The upper bound}\label{sec:upper}
We return to smooth bounded skew streams $U\in C_b^\infty$ and their classical kernels, now under the structural decomposition. We combine the exponential-weight method with Nash iteration; see \cite{carlen1987upper} for the symmetric semigroup setting. A sharp weighted $L^2$ estimate on the middle interval and smoothing on two short endpoint intervals lead to the kernel bound. Both the weighted estimate and the endpoint smoothing are obtained directly from energy estimates and Nash's inequality, without using \eqref{eq:aronson}.

For the weighted upper estimate, we construct a bounded, time-independent Lipschitz weight whose increment approximates the endpoint distance and whose drift contribution can be made small.

\begin{lemma}\label{lem:weights}
Suppose that $u=\div U$ satisfies Assumption \ref{ass:decomp} on $(0,T)$. For every $\delta>0$, one can choose $\sigma_\delta=\sigma(d,\delta)>0$ and $K_\delta=K(d,\delta, \|F\|_{L^{d,\infty}})>0$ such that the following holds: for every $x,y\in\mR^d$, there is a bounded real Lipschitz weight $\psi$ such that
\begin{equation}\label{eq:weight}
\|\nabla\psi\|_\infty\leq1,
\quad
\psi(y)-\psi(x)\geq |x-y|-\delta,
\end{equation}
and, for almost every $t\in(0,T)$ and every real $\varphi\in C_c^\infty(\mR^d)$,
\begin{equation}\label{eq:drift}
2\left|\cE_t^a(\psi\varphi,\varphi)\right|
\leq
\delta\|\nabla\varphi\|_2\|\varphi\|_2
+c_\delta(t)\|\varphi\|_2^2, 
\end{equation}
where $c_\delta(t):=K_\delta+\beta_{\sigma_\delta}(t)\in L^1(0,T)$. For a family sharing $F$ and the bounds $\beta_\sigma$, the same $\psi,c_\delta$ may be used for every member at each fixed $\delta,x,y$. 
\end{lemma}

\const{\Cdist} \const{\CFzero}
\begin{proof}
Fix $T>0$ and $\delta>0$. Let $\eta,\sigma,R>0$ be parameters to be chosen below, independently of time and of $x,y$.  Write $N:=\|F\|_{L^{d,\infty}}$.  By \eqref{eq:decomp} and the skew-potential characterization in Appendix \ref{sec:aux}, we may write
\[
u=v+\div A_\sigma+b_\sigma,
\quad
A_\sigma^{\mathsf T}=-A_\sigma,
\quad
\esssup_{0<t<T}
\|A_\sigma(t)\|_{{\rm BMO}}\leq C\sigma.
\]
$\beta_\sigma(t)=\|b_\sigma(t)\|_\infty \in L^1(0,T)$.

Let $O:=O_{\eta,R}$ be the open set given by Lemma \ref{lem:flatten} for the fixed envelope $F$. This set is chosen only depends on $d,F,\eta,R$, its dependence on $d$ and $F$ is omitted in the notation. In particular, $O$ is independent of time, the direction $e$, and the endpoints. For $x,y\in\mR^d$ with $x\ne y$, choose $e:=(y-x)/|y-x|$ and take the bounded Lipschitz function $\psi:=\psi_{e,x,y}$ from Lemma \ref{lem:flatten}, which satisfies
\[
\|\nabla\psi\|_\infty\leq1,
\quad
\nabla\psi(x)=0\quad\text{a.e. } x\in O ~\text{ and }~ \psi(y)-\psi(x)\geq |x-y|-C_{\Cdist}\Big(\frac{N}{\eta}\Big)^dR.
\]
The constant in this estimate is the one in \eqref{eq:wl-cal}.

Fix a real $\varphi\in C_c^\infty(\mR^d)$. The decomposition of $u$ and integration by parts imply that  
\[
\begin{aligned}
-2\cE_t^a(\psi\varphi,\varphi)
={}&\int v(t,z)\cdot\nabla\psi(z)\varphi(z)^2\,\d z\\
&-2\cE_{A_\sigma(t)}^a(\psi\varphi,\varphi) +\int b_\sigma(t,z)\cdot\nabla\psi(z)\varphi(z)^2\,\d z.
\end{aligned}
\]

We estimate these three terms separately.  Since $\nabla \psi$ vanishes on $O$, we obtain from \eqref{eq:wl-trace}
\[
\begin{aligned}
\left|\int v(t)\cdot\nabla\psi\,\varphi^2\,\d z\right|\leq\int_{O^c}F\varphi^2\,\d z\leq C_{\CFzero}\eta\|\nabla\varphi\|_2\|\varphi\|_2
+C_{\CFzero}\frac{\eta}{R}\|\varphi\|_2^2.
\end{aligned}
\]
For the small skew part, Lemma \ref{lem:clms} yields
\[
\begin{aligned}
2|\cE_{A_\sigma(t)}^a(\psi\varphi,\varphi)|\leq C\|A_\sigma(t)\|_{{\rm BMO}}
\|\nabla\psi\|_\infty\|\nabla\varphi\|_2\|\varphi\|_2 \leq C\sigma\|\nabla\varphi\|_2\|\varphi\|_2.
\end{aligned}
\]
Finally,
\[
\left|\int b_\sigma(t)\cdot\nabla\psi\,\varphi^2\,\d z\right|
\leq\beta_\sigma(t)\|\varphi\|_2^2.
\]
Combining the three bounds, we obtain that
\const{\CFone}
\[
2|\cE_t^a(\psi\varphi,\varphi)|
\leq C_{\CFone}(\eta+\sigma)\|\nabla\varphi\|_2\|\varphi\|_2
+\Big(C_{\CFzero}\frac{\eta}{R}+\beta_\sigma(t)\Big)\|\varphi\|_2^2.
\]

Now take $\eta=\sigma=\delta/(2C_{\CFone})$, so that $C_{\CFone}(\eta+\sigma)\leq\delta$. Set $R:=\delta/(1+C_{\Cdist}(N/\eta)^d)>0$, then $C_{\Cdist}(N/\eta)^dR\leq\delta$. With these choices, we obtain \eqref{eq:weight} and \eqref{eq:drift}, with
\[
K=C_{\CFzero}\frac{\eta}{R}>0, \quad c_\delta(t):=K+\beta_\sigma(t)\in L^1(0,T). 
\]
The constants $C_{\Cdist},C_{\CFzero},C_{\CFone}$ depend only on $d$. Thus, $\sigma$ depends only on $d,\delta$, and $R$ and $K$ depend only on $d,\delta,N$. The chosen parameters and remainder are independent of $x,y$, and the weight depends only on $d,\delta,F,x,y$. Thus, the same choices work for a family sharing the decomposition data. 
\end{proof}

\begin{proof}[Proof of Theorem \ref{thm:upper}]
Fix $\delta>0$ and $x,y\in\mR^d$, and choose $\psi,c_\delta$ from Lemma \ref{lem:weights}. Write
\[
C_\delta(t):=\int_0^t c_{\delta}(s)\,\d s,\quad
A_\delta:=1+\frac{\delta^2}{2},\quad
\kappa_\delta:=\frac{(\delta+2)^2}{4},\quad
D_\delta:=A_\delta+6\kappa_\delta.
\]

\medskip
\noindent \emph{Step 1: Weighted energy estimates.} Fix $\alpha\geq0$. Write $M_\vartheta f:=\e^\vartheta f$ and
\[
Q_{a,b}:=M_{-\alpha\psi}P_{a,b}M_{\alpha\psi},\quad
\cE_t^{\alpha,\psi}(v,w)
:=\cE_t(\e^{\alpha\psi}v,\e^{-\alpha\psi}w).
\]
For $0\leq a<b\leq T$, set $\tau:=b-a$ and $q(s):=Q_{b-s,b}f$, $0\leq s\leq\tau$. We first take $f=\e^{-\alpha\psi}g$ with $0\leq g\in C_c^\infty$. Then $\e^{\alpha\psi}q(s)=P_{b-s,b}g$ solves the classical backward equation. The maximum principle implies
\[
0\leq q(s)\leq\e^{\alpha(\sup\psi-\inf\psi)}\|f\|_\infty.
\]
For each fixed $2\leq r<\infty$, differentiation and integration by parts yield, for almost every $s\in(0,\tau)$, the energy identity
\begin{equation}\label{eq:energy_r}
\frac1r\frac{\d}{\d s}\|q(s)\|_r^r
=-\cE_{b-s}^{\alpha,\psi}(q(s),q(s)^{r-1}).
\end{equation}
Only first derivatives of the bounded Lipschitz weight occur; spatial cutoffs and bounded $H^1$ approximation justify the integrations. The chosen initial data are dense in the nonnegative cone of $L^2$, since multiplication by $\e^{-\alpha\psi}$ is bounded and invertible. Positivity, $|Q_{a,b}f|\leq Q_{a,b}|f|$, and density therefore extend the ensuing norm bounds to all $L^2$ data.

\medskip

Next we study the energy identity both for $r=2$ and for the powers used in Nash iteration below. For a bounded nonnegative $v\in H^1$ and $r\geq2$, expanding the local form, we obtain
\begin{equation}\label{eq:power}
\begin{aligned}
-\cE_t^{\alpha,\psi}(v,v^{r-1})
=&-\frac{4(r-1)}{r^2}\|\nabla v^{\frac{r}{2}} \|_2^2
-\frac{2\alpha(r-2)}r
\int v^{\frac{r}{2}}\,\nabla\psi\cdot\nabla v^{\frac{r}{2}}\,\d x\\
&+\alpha^2\int |\nabla\psi|^2 v^r\,\d x
-2\alpha\cE_t^a(\psi v^{\frac{r}{2}}, v^{\frac{r}{2}}).
\end{aligned}
\end{equation}
Since $U$ is bounded and multiplication by $\psi$ is continuous on $H^1$, density extends \eqref{eq:drift} to real $H^1$ tests. In particular, it applies to $z=v^{r/2}\in H^1$; bounded $H^1$ approximation also justifies \eqref{eq:power}.

For $r=2$, \eqref{eq:drift} and \eqref{eq:power}, together with $\alpha\delta\|\nabla v\|_2\|v\|_2\leq\frac12\|\nabla v\|_2^2+\frac12\delta^2\alpha^2\|v\|_2^2$, imply that
\begin{equation*}
-\cE_t^{\alpha,\psi}(v,v)
\leq-\frac12\|\nabla v\|_2^2
+(A_\delta\alpha^2+\alpha c_\delta(t))\|v\|_2^2.
\end{equation*}
By energy identity \eqref{eq:energy_r} with $r=2$ and Gronwall's inequality, we obtain
\begin{equation}\label{eq:weighted-l2}
\|Q_{a,b}\|_{2\to2}
\leq\exp\Big\{A_\delta\alpha^2(b-a)
+\alpha\int_a^b c_\delta(t)\,\d t\Big\}.
\end{equation}

\medskip
\noindent \emph{Step 2: Nash iteration.} 
In order to remove the time remainder, we consider 
\[
\widehat q(s):=
\exp\Big\{-\alpha^2s-\alpha\int_0^s c_\delta(b-\rho)\,\d\rho\Big\}q(s)
\]
instead of $q$ directly. For fixed $r\geq2$, using $\|\nabla\psi\|_\infty\leq1$ and \eqref{eq:drift} in \eqref{eq:power}, then applying Young's inequality, we obtain
\[
\begin{aligned}
\frac1r\frac{\d}{\d s}\|q(s)\|_r^r 
\leq & -\frac{4(r-1)}{r^2}\|\nabla q(s)^{\frac{r}{2}}\|_2^2\\
&+\alpha(\delta+2)\|\nabla q(s)^{\frac{r}{2}}\|_2\|q(s)^{\frac{r}{2}}\|_2
+(\alpha^2+\alpha c_\delta(b-s))\|q(s)^{\frac{r}{2}}\|_2^2\\
\leq&-\frac{2(r-1)}{r^2}\|\nabla q(s)^{\frac{r}{2}}\|_2^2
+\bigl(\alpha^2+\alpha c_\delta(b-s)+\kappa_\delta r\alpha^2\bigr)\|q(s)^{\frac{r}{2}}\|_2^2.
\end{aligned}
\]
In the second step, we used Young's inequality to absorb the mixed term with gradient coefficient $2(r-1)/r^2$. The remaining coefficient is bounded by $\kappa_\delta r\alpha^2$ as $r^2/[8(r-1)]\leq r/4$ for $r\geq2$.

Applying the scalar integrating factor to the absolutely continuous function $\|q(s)\|_r^r$, which requires only $c_\delta\in L^1(0,T)$, we obtain
\begin{equation*}
\begin{aligned}
\frac{\d}{\d s}\|\widehat q(s)\|_r^r
&\leq-\frac{2(r-1)}r\|\nabla(\widehat q(s)^{r/2})\|_2^2
+\kappa_\delta r^2\alpha^2\|\widehat q(s)\|_r^r\\
&\leq-\|\nabla(\widehat q(s)^{r/2})\|_2^2
+\kappa_\delta r^2\alpha^2\|\widehat q(s)\|_r^r,
\quad r\geq2.
\end{aligned}
\end{equation*}
By the Nash inequality 
\[
\|w\|_2^{2+4/d}\leq C\|\nabla w\|_2^2\|w\|_1^{4/d},
\quad w\in H^1\cap L^1.
\]
If $r\geq4$ and $\|\widehat q(s)\|_{r/2}\leq M$ on an interval, then $Y:=\|\widehat q\|_r^r$ satisfies
\[
Y'\leq-C^{-1}M^{-2r/d}Y^{1+2/d}
+\kappa_\delta r^2\alpha^2Y.
\]
On a subinterval $[t-\ell,t]$, multiply by $\exp\{-\kappa_\delta r^2\alpha^2(s-t+\ell)\}$ and integrate the differential inequality for the negative power $-2/d$.  We thus obtain
\begin{equation}\label{eq:nash-step}
\|\widehat q(t)\|_r
\leq(C/\ell)^{d/(2r)}
\e^{\kappa_\delta r\alpha^2\ell}
\sup_{t-\ell\leq s\leq t}\|\widehat q(s)\|_{r/2}.
\end{equation}
Now set $r_k:=2^{k+1}$, $s_k:=\tau(1-4^{-k})$, and $M_k:=\sup_{s_k\leq s\leq\tau}\|\widehat q(s)\|_{r_k}$ for $k\geq0$. For $k\geq1$, use \eqref{eq:nash-step} on sliding intervals of length $\ell_k:=s_k-s_{k-1}=3\tau4^{-k}$ to obtain
\[
M_k\leq(C/\ell_k)^{d/(2r_k)}
\e^{\kappa_\delta r_k\alpha^2\ell_k}M_{k-1}.
\]
Since
\[
\sum_{k\geq1}\frac1{r_k}=\frac12,\quad
\sum_{k\geq1}\frac{k}{r_k}=1,\quad
\sum_{k\geq1}r_k\ell_k=6\tau,
\]
By iteration and \eqref{eq:weighted-l2}, we obtain
\[
\|\widehat q(\tau)\|_\infty
\leq C\tau^{-d/4}\e^{6\kappa_\delta\alpha^2\tau}M_0,
\quad
M_0\leq\e^{(A_\delta-1)\alpha^2\tau}\|f\|_2.
\]
By the definition of $q$ and $\widehat{q}$, undoing the scalar factor, we get
\begin{equation}\label{eq:smoothing}
\|Q_{a,b}\|_{2\to\infty}
\leq C(b-a)^{-d/4}
\exp\Big\{D_\delta\alpha^2(b-a)
+\alpha\int_a^b c_\delta(t)\,\d t\Big\}.
\end{equation}

For the adjoint, put $q^*(s):=M_{\alpha\psi}P_{a,a+s}^*M_{-\alpha\psi}f$, so that $q^*(\tau)=Q_{a,b}^*f$. The classical forward equation for $P_{a,a+s}^*$ has generator $L_{a+s}^*=\Delta-u(a+s)\cdot\nabla$, since $\div u=0$. Thus the same calculation applies with $\cE^a$ and $\psi$ both replaced by their negatives, and with $c_\delta(a+s)$ in place of $c_\delta(b-s)$. The bound \eqref{eq:drift} is unchanged by these sign replacements. The resulting $2\to\infty$ estimate for $Q_{a,b}^*$ and duality yield
\begin{equation}\label{eq:weighted-1-2}
\|Q_{a,b}\|_{1\to2}
\leq C(b-a)^{-d/4}
\exp\Big\{D_\delta\alpha^2(b-a)
+\alpha\int_a^b c_\delta(t)\,\d t\Big\}.
\end{equation}

\medskip
\noindent \emph{Step 3: Optimization.} Take $0<t<T\wedge\frac{1}{2}$ and set $\tau:=t^2$, so that $0<\tau<t-\tau<t$. For the same fixed weight, the evolution property implies $Q_{0,t}=Q_{0,\tau}Q_{\tau,t-\tau}Q_{t-\tau,t}$. Using \eqref{eq:smoothing}, \eqref{eq:weighted-l2}, and \eqref{eq:weighted-1-2} on the three intervals, we obtain
\[
\begin{aligned}
\|Q_{0,t}\|_{1\to\infty}
\leq& \|Q_{0,\tau}\|_{2\to\infty}
\|Q_{\tau,t-\tau}\|_{2\to2}
\|Q_{t-\tau,t}\|_{1\to2}\\
\leq & C\tau^{-d/4}
\exp\Big\{D_\delta\alpha^2\tau+\alpha\int_0^\tau c_\delta(s)\,\d s\Big\}\\
&\cdot \exp\Big\{A_\delta\alpha^2(t-2\tau)
+\alpha\int_\tau^{t-\tau}c_\delta(s)\,\d s\Big\} \\
&\cdot C\tau^{-d/4} \exp\Big\{D_\delta\alpha^2\tau+\alpha\int_{t-\tau}^t c_\delta(s)\,\d s\Big\}\\
\leq & C\tau^{-d/2}
\exp\{\alpha^2B_\delta(t)+\alpha C_\delta(t)\}.
\end{aligned}
\]
Here the three integrals of $c_\delta$ add to $C_\delta(t)$, and
\[
B_\delta(t):=A_\delta(t-2\tau)+2D_\delta\tau.
\]
For a fixed weight, the kernel of $Q_{0,t}$ is $\e^{\alpha(\psi(w)-\psi(z))}p_{0,t}(z,w)$. The operator bound controls this kernel for almost every $(z,w)$; continuity of $p_{0,t}$ and of $\psi$ extends the bound to every $(z,w)$. We may therefore evaluate it at the chosen endpoints $(x,y)$ and use \eqref{eq:weight} to obtain
\begin{equation*}
p_{0,t}(x,y)\leq C\tau^{-d/2}
\exp\Big\{-\alpha(|x-y|-\delta-C_\delta(t))
+\alpha^2B_\delta(t)\Big\}.
\end{equation*}
Since $\tau=t^2$, we have
\[
B_\delta(t)=A_\delta t+12\kappa_\delta t^2
=\Big(1+\frac{\delta^2}{2}\Big)t+3(\delta+2)^2t^2.
\]
Minimizing over $\alpha\geq0$, with $\alpha=(|x-y|-e_\delta(t))_+/(2B_\delta(t))$, proves \eqref{eq:weighted-opt}. The prefactor comes only from Nash's inequality and its iteration, so it depends only on $d$. If a family shares the decomposition data, we obtain common weights and remainders from Lemma \ref{lem:weights}, so every estimate above is uniform over that family.
\end{proof}
\section{Sufficient classes and applications}\label{sec:applications}
We identify concrete classes satisfying \eqref{eq:sector} and Assumption \ref{ass:decomp}, prove Corollary \ref{cor:classes} with its time-independent weak-$L^d$ conclusion, and then apply its time-dependent statement to the vorticity equation.

\subsection{Classes satisfying the sector condition}

We first identify drift classes satisfying \eqref{eq:sector}, with quantitative constants that can be used in the lower bound of Theorem \ref{thm:lower}.

\begin{lemma}\label{lem:sector-sum}
Let $d\geq2$, $T>0$, and let $u$ be divergence-free on $(0,T)\times\mR^d$. Then \eqref{eq:sector} holds in each of the following cases:
\begin{enumerate}[(i)]
\item $u=a+b$, where $a$ and $b$ are divergence-free and
\[
A:=\esssup_{0<t<T}\|a(t)\|_{{\rm BMO}^{-1}}<\infty,
\quad B(t):=\|b(t)\|_\infty\in L^2(0,T).
\]
One may take $M=CA$ and $m(t)=CB(t)$.
\item
\begin{equation}\label{eq:mixed}
u\in L^q(0,T;L^p(\mR^d)),\quad
d\leq p\leq\infty,\quad \frac2q+\frac dp\leq1.
\end{equation}
If $p=d$, then $q=\infty$, and one may take
\[
M=C\|u\|_{L^\infty(0,T;L^d)},\quad m=0.
\]
And if $p>d$, one may take
\begin{equation*}
M=1,\quad m(t)=C\|u(t)\|_p^{\alpha},\quad
\alpha:=\frac{p}{p-d}. 
\end{equation*}
\end{enumerate}
For compact smooth tests, the form is understood through $\cE_t^a(f,g)=-\langle u(t),g\nabla f\rangle$; this agrees with \eqref{eq:local-forms} whenever a locally integrable stream is specified. 
\end{lemma}

\begin{proof}
For (i), apply the skew-potential characterization in Appendix \ref{sec:aux} to $a$. By the Hardy--BMO estimate \eqref{eq:hbmo} and Cauchy--Schwarz for $b$, we have
\[
\begin{aligned}
|\cE_t^a(f,g)|
&\leq |\langle a(t),g\nabla f\rangle|
      +\left|\int b(t)\cdot\nabla f\,g\right|\\
&\leq C\|\nabla f\|_2
\bigl(A\|\nabla g\|_2+B(t)\|g\|_2\bigr).
\end{aligned}
\]
This is \eqref{eq:sector} with the stated constants.

For (ii) at $p=d$, we obtain from Lemma \ref{lem:hodge} a skew stream with ${\rm BMO}$ norm bounded by $C\|u(t)\|_d$. The estimate \eqref{eq:hbmo} therefore yields
\[
|\cE_t^a(f,g)|
\leq C\|u\|_{L^\infty(0,T;L^d)}
\|\nabla f\|_2\|\nabla g\|_2,
\]
including when $d=2$.

For $d<p<\infty$, by the H\"older, Gagliardo--Nirenberg and Young inequalities, we obtain
\[
\begin{aligned}
|\cE_t^a(f,g)|
&\leq\|u(t)\|_p\|\nabla f\|_2\|g\|_{2p/(p-2)}\\
&\leq C\|u(t)\|_p\|\nabla f\|_2
\|\nabla g\|_2^{d/p}\|g\|_2^{1-d/p}\\
&\leq\|\nabla f\|_2
\bigl(\|\nabla g\|_2+C\|u(t)\|_p^\alpha\|g\|_2\bigr).
\end{aligned}
\]
For $p=\infty$, the conclusion follows directly from Cauchy--Schwarz. Since $q\geq2\alpha$, H\"older's inequality in time yields
\begin{equation*}
\|m\|_{L^2(0,T)}
\leq CT^{1/2-\alpha/q}
\|u\|_{L^q(0,T;L^p)}^\alpha.
\end{equation*}
\end{proof}

\subsection{Classes satisfying the structural decomposition}

\begin{proposition}\label{prop:class}
For divergence-free drift $u$. Assumption \ref{ass:decomp} holds for every time-independent divergence-free field $u\in L^{d,\infty}(\mR^d)$ and in each of the following cases:
\begin{enumerate}[(i)]
\item $u\in L^q(0,T;L^p(\mR^d))$ for every $T>0$, where $d<p\leq\infty$ and $d/p+1/q\leq1$;
\item $d\geq3$ and $u$ is a weakly measurable divergence-free vector of signed measures satisfying the Kato condition: for every $T<\infty$,
\begin{equation}\label{eq:kato}
\lim_{r\downarrow0}
\esssup_{0<t<T}\sup_{x\in\mR^d}
\int_{|x-y|<r}|x-y|^{1-d}\,\nu_t(\d y)=0.
\end{equation}
\end{enumerate}
\end{proposition}

\begin{proof}
For a time-independent weak-$L^d$ field, take $v=u$, $F=|u|$, and $a_\sigma=b_\sigma=0$. Fix $\sigma>0$.  In each case we take $v:=0$ and $F:=0$, and construct divergence-free fields $a_\sigma,b_\sigma$ such that
\[
u=a_\sigma+b_\sigma,\quad
\esssup_{t>0}
\|a_\sigma(t)\|_{{\rm BMO}^{-1}}\leq\sigma,\quad
\int_0^T\|b_\sigma(t)\|_\infty\,\d t<\infty
\quad(T>0).
\]
\medskip
\noindent \emph{(i) The mixed Lebesgue class.} For $p=\infty$, take $a_\sigma=0$ and $b_\sigma=u$. For $d<p<\infty$, we use the interpolation decomposition between $L^\infty_tL^d_x$ and $L^1_tL^\infty_x$, together with $L^d\hookrightarrow{\rm BMO}^{-1}$. Explicitly, put $M(t):=\|u(t)\|_p$ and, for $\varepsilon>0$, set
\[
k_\varepsilon(t):=\varepsilon^{-d/(p-d)}M(t)^{p/(p-d)},\qquad
f_\varepsilon:=u\mathbf1_{\{|u|>k_\varepsilon\}},\quad
 g_\varepsilon:=u-f_\varepsilon.
\]
Then, for almost every $t$ with $M(t)>0$,
\[
\|f_\varepsilon(t)\|_d^d
\leq k_\varepsilon(t)^{d-p}M(t)^p=\varepsilon^d,
\qquad
\|g_\varepsilon(t)\|_\infty\leq k_\varepsilon(t).
\]
Thus $\|f_\varepsilon\|_{L^\infty_tL^d_x}\leq\varepsilon$ and $g_\varepsilon\in L^1(0,T;L^\infty_x)$, since $q\geq p/(p-d)$.

To retain divergence-free summands, write $S_h:=\e^{h\Delta}$. By the embedding $L^d\hookrightarrow{\rm BMO}^{-1}$ and the heat kernel bounds, we have
\[
\begin{aligned}
\|(I-S_h)f\|_{{\rm BMO}^{-1}}&\leq C\|f\|_d,
&\quad \|S_hf\|_\infty&\leq Ch^{-1/2}\|f\|_d,\\
\|(I-S_h)g\|_{{\rm BMO}^{-1}}&\leq C\sqrt h\|g\|_\infty,
&\quad \|S_hg\|_\infty&\leq\|g\|_\infty.
\end{aligned}
\]
For the third estimate, use $g-S_hg=-\int_0^h\Delta S_sg\,\d s$ and $\|\nabla S_sg\|_\infty\leq Cs^{-1/2}\|g\|_\infty$. Taking $h(t):=(\varepsilon/k_\varepsilon(t))^2$ and
\[
b_\sigma(t):=S_{h(t)}u(t),\qquad a_\sigma(t):=u(t)-b_\sigma(t),
\]
we obtain $\|a_\sigma(t)\|_{{\rm BMO}^{-1}}\leq C\varepsilon$ and $\|b_\sigma(t)\|_\infty\leq Ck_\varepsilon(t)$. Choose $\varepsilon=\sigma/C$ with $C$ fixed large enough, and set both summands to zero when $M(t)=0$. They are measurable by strong $L^p$ measurability of $u$ and continuity of $S_h$ for $h>0$, and divergence-free because $S_h$ commutes with divergence. This proves the required decomposition. Here $C$ depends only on $d$.

\medskip
\noindent \emph{(ii) The Kato measure class.} Put $\mu(t):=u(t)$, denote its components by $\mu_1(t),\ldots,\mu_d(t)$, and set
\[
\nu_t:=\sum_j|\mu_j(t)|.
\]
Assume that $T=1$. Define
\[
\kappa(r):=
\esssup_{t\in [0,1]}\sup_{x\in\mR^d}
\int_{|x-y|<r}|x-y|^{1-d}\,\nu_t(\d y).
\]
By \eqref{eq:kato}, $\kappa(r)\to0$ as $r\downarrow0$. 

For $\lambda>0$, let
\[
G_\lambda(z):=\int_0^\infty
\e^{-\lambda s}(4\pi s)^{-d/2}\e^{-|z|^2/(4s)}\,\d s.
\]
Then $(\lambda-\Delta)G_\lambda=\delta_0$ in distributions. For $d\geq3$, by definition, we have 
\[
|\nabla G_\lambda(z)|+\sqrt\lambda G_\lambda(z)
\leq C|z|^{1-d}\e^{-\sqrt\lambda|z|/C}
\quad(z\ne0).
\]
Set
\[
\eta_{\lambda}:=
\esssup_{t\in [0,1]}\sup_{x\in\mR^d}
\int
\bigl(|\nabla G_\lambda(x-y)|+\sqrt\lambda G_\lambda(x-y)\bigr)
\,\nu_t(\d y).
\]
We claim that
\begin{equation}\label{eq:kato_eta}
\eta_{\lambda}\leq C\kappa(2\lambda^{-1/2})
\to0, \quad \lambda\to\infty.  
\end{equation}
Put $\ell:=\lambda^{-1/2}$ and fix $t$ for which the Kato bound holds. We estimate the integral over $B_{2\ell}(x)$ and $B^c_{2\ell}(x)$ separately. On $B_{2\ell}(x)$, the kernel bound above and the definition of $\kappa$ imply that the contribution is at most $C\kappa(2\ell)$.

For the complement, first note that the Kato bound controls the mass of every ball:
\[
\nu_t(B_r(z))
\leq r^{d-1}\int_{B_r(z)}|z-y|^{1-d}\,\nu_t(\d y)
\leq r^{d-1}\kappa(r).
\]
For $k\geq2$, let $A_k(x):=\{y:k\ell\leq|x-y|<(k+1)\ell\}$. This annulus can be covered by at most $Ck^{d-1}$ balls of radius $\ell/4$, so
\[
\nu_t(A_k(x))\leq Ck^{d-1}\ell^{d-1}\kappa(\ell/4).
\]
The kernel on $A_k(x)$ is bounded by $C(k\ell)^{1-d}\e^{-k/C}$. Hence
\[
\begin{aligned}
&\int_{A_k(x)}
\bigl(|\nabla G_\lambda(x-y)|+\sqrt\lambda G_\lambda(x-y)\bigr)
\,\nu_t(\d y)\\
\leq& C(k\ell)^{1-d}\e^{-k/C}\nu_t(A_k(x))
\leq C\e^{-k/C}\kappa(2\ell).
\end{aligned}
\]
Summing over $k\geq2$ and adding the contribution from $B_{2\ell}(x)$ proves the claimed bound, uniformly in $x$ and almost every $t$. 
%In particular, the convolutions with $G_\lambda$ and $\nabla G_\lambda$ below converge absolutely for sufficiently large $\lambda$. The ball-mass estimate also implies that $\nu_t$ has at most polynomial growth: fix a small $r$ with $\kappa(r)<\infty$ and cover $B_R(0)$ by at most $C(1+R/r)^d$ balls of radius $r$. Thus $\mu(t)$ is a tempered distribution, and $G_\lambda*\mu(t)$ agrees with the resolvent defined by the Fourier multiplier $(\lambda+|\xi|^2)^{-1}$. We may therefore use the resolvent identity and commute spatial derivatives with the resolvent.

For almost every $t\in [0,1]$, define
\[
w_{\lambda}(t):=G_\lambda*\mu(t),\quad
A_{\lambda,ij}(t):=
\partial_j(w_{\lambda}(t))_i-\partial_i(w_{\lambda}(t))_j,\quad
b_\lambda(t):=\lambda w_{\lambda}(t).
\]
The derivatives of $w_{\lambda}(t)$ are given by the absolutely convergent convolutions with $\nabla G_\lambda$. Since $\div\mu(t)=0$, the resolvent commutes with divergence and $\div w_{\lambda}(t)=0$. With our convention $(\div A)_j=\sum_i\partial_iA_{ij}$, we thus obtain
\[
(\div A_\lambda)_j
=\partial_j\div w_{\lambda}(t)-\Delta(w_{\lambda}(t))_j
=-\Delta(w_{\lambda}(t))_j.
\]
Thus, $A_\lambda$ is skew-symmetric, and on putting $a_\lambda:=\div A_\lambda$, the resolvent identity yields 
\[
\mu(t)=a_\lambda(t)+b_\lambda(t).
\]
Both summands are divergence-free.  By definition, 
\[
\begin{aligned}
\|a_\lambda(t)\|_{{\rm BMO}^{-1}}
\leq C\|A_\lambda(t)\|_{{\rm BMO}}
\leq C\|\nabla w_{\lambda}(t)\|_\infty
\leq C\eta_{\lambda}, \quad 
\|b_\lambda(t)\|_\infty
\leq C\sqrt\lambda\,\eta_{\lambda}.
\end{aligned}
\]
Combining this and \eqref{eq:kato_eta}, we obtain the desired assertion. 
\end{proof}

\begin{remark}
Proposition \ref{prop:class}(ii) uses the Kato condition in Chen--Fang--Zhang \cite{CFZ19small}, with only the normalization changed from $\frac12\Delta$ to $\Delta$.  Their general drift need not be divergence-free. For a divergence-free member in $L^\infty_t{\rm BMO}^{-1}_x$, Proposition \ref{prop:class} and Corollary \ref{cor:classes} apply; \cite{CFZ19small} treats the general case by a different construction.
\end{remark}

\subsection{Euclidean Varadhan asymptotics}

By Proposition \ref{prop:bounds}, we obtain a common integrable remainder for the smooth approximations, allowing the upper bound to pass to the limiting kernel. We combine this with the lower bound to prove Corollary \ref{cor:classes}, and then turn to the vorticity equation.

\begin{proof}[Proof of Corollary \ref{cor:classes}]
We first prove the general time-dependent statement. By the skew-potential characterization in Appendix \ref{sec:aux}, we can choose a jointly Borel skew stream $U\in L^\infty(0,T;{\rm BMO})$ with $u=\div U$. Normalize $U$ by $U_{B_1(0)}(t)=0$, extend it by zero outside $(0,T)$, and choose the approximations \eqref{eq:bandpass}. They satisfy \eqref{eq:approx} with $N_*\leq CN_U$, and Theorem \ref{thm:feller} implies locally uniform convergence of their kernels $p^n$ to a positive continuous kernel $p$ on $\{s<t\}$.

Fix $\delta>0$. By Assumption \ref{ass:decomp} and Proposition \ref{prop:bounds}(ii), the kernels $p^n$ satisfy \eqref{eq:weighted-opt} with one nonnegative $c_\delta\in L^1(0,T)$ independent of $n,x,y$. Passing to the limit $n\to\infty$ at fixed $0<t<T\wedge\frac12$, we obtain the same upper estimate for $p$. Moreover,
\[
e_\delta(t)\to\delta,\quad
\frac{B_\delta(t)}t\to1+\frac{\delta^2}{2}
\quad\text{as }t\downarrow0.
\]
Thus, for every compact $K\subset\mR^d$,
\[
\limsup_{t\downarrow0}\sup_{x,y\in K}
\left\{t\log p_{0,t}(x,y)
+\frac{(|x-y|-\delta)_+^2}{4(1+\delta^2/2)}\right\}\leq0.
\]
Here the prefactor contributes $t\log(Ct^{-d})\to0$, and the preceding limits are uniform in the endpoints. Letting $\delta\downarrow0$, we obtain the locally uniform upper Varadhan bound. The matching lower bound follows from Corollary \ref{cor:bmo-lower}, proving \eqref{eq:var}.

Finally, let $u\in L^{d,\infty}(\mR^d;\mR^d)$ be time-independent and divergence-free. Lemma \ref{lem:hodge} implies $u\in{\rm BMO}^{-1}$, while
\[
v=u,\qquad F=|u|,\qquad a_\sigma=b_\sigma=0
\]
verifies Assumption \ref{ass:decomp} on every finite time interval. The general statement therefore applies, with no smallness restriction on $\|u\|_{L^{d,\infty}}$.
\end{proof}

For the vorticity equation, we verify the two hypotheses in the time-dependent statement of Corollary \ref{cor:classes}: a uniform ${\rm BMO}^{-1}$ bound and the structural decomposition.

\begin{proof}[Proof of Corollary \ref{cor:ns}]
Weak-* continuity and the uniform boundedness principle imply $M_\rho:=\sup_{0\leq t\leq T}\|\rho_t\|_{\mathrm{TV}}<\infty$. Let
\[
G(x):=\frac1{2\pi}\log|x|,
\quad
\psi_\rho(t):=G*\rho_t,
\quad
B_\rho(t):=
\begin{pmatrix}
0&\psi_\rho(t)\\
-\psi_\rho(t)&0
\end{pmatrix}.
\]
The convolution $G*\rho_t$ is understood modulo spatial constants. Then $u_\rho=\nabla^\perp\psi_\rho=\div B_\rho$.  Since $G\in{\rm BMO}(\mR^2)$, by translation invariance and convexity of the BMO seminorm, we have, modulo spatial constants,
\[
\sup_{0\leq t\leq T}\|\psi_\rho(t)\|_{{\rm BMO}}
\leq M_\rho\|G\|_{{\rm BMO}}.
\]
Thus, $u_\rho\in L^\infty(0,T;{\rm BMO}^{-1})$. Extend $B_\rho$ by zero outside $[0,T]$ when applying Theorem \ref{thm:feller}. The Aronson estimate implies $\|\rho_t\|_{4/3}\leq Ct^{-1/4}$. By the Sobolev inequality,
\[
\|u_\rho(t)\|_4\leq C\|\rho_t\|_{4/3}
\leq Ct^{-1/4}.
\]
In particular, $u_\rho\in L^2(0,T;L^4_x)$. Proposition \ref{prop:class}(i), with $d=2$, $p=4$ and $q=2$, now verifies Assumption \ref{ass:decomp}. Together with the uniform ${\rm BMO}^{-1}$ bound above, this establishes the hypotheses of Corollary \ref{cor:classes} and proves the result. 
\end{proof}
 
\section{A singular path law with Euclidean Varadhan asymptotics}\label{sec:path}

In this section, we construct a tangential divergence-free vector field $u\in L^d(\mR^d)$ such that $\int_0^1 |u(\xi_t)|^2\,\d t=\infty$, $\mW_0$-a.s.  If the diffusion law and $\mW_0$ were not mutually singular, their common part could be normalized to a probability measure $Q$ with finite relative entropy with respect to both laws.  F\"ollmer's theorem (Lemma \ref{lem:follmer} below) then implies $\int_0^1 |u(\xi_t)|^2\,\d t<\infty$, $Q$-a.s.  This contradicts $Q\ll\mW_0$.

Since the same drift is time-independent and in $L^d$, Lemma \ref{lem:hodge} places it in ${\rm BMO}^{-1}$, and it satisfies Assumption \ref{ass:decomp}. Corollary \ref{cor:classes} therefore implies the Euclidean Varadhan formula for its transition density.

\medskip

Let $d\geq 2$, and take
\[
Jx:=(-x_2,x_1,0,\ldots,0) 
\]
and
\[
\alpha_d:=\frac{2}{3} \1_{\{d=2\}}+\frac12 \1_{\{d\geq3\}}.
\]
Define
\[
\beta(r):=
\frac{\1_{\{ 0<r<\e^{-2} \}}}{r^2 \log^{\alpha_d}(1/r)}, 
\quad
u(x):=\beta(|x|)Jx.
\]
Let $c_d:=\int_{\mS^{d-1}}|J\theta|^d\,\d\theta>0$. Then
\[
\begin{aligned}
\int|u(x)|^d\,\d x
&=
c_d\int_0^{\e^{-2}}
\frac{\d r}{r\bigl(\log(1/r)\bigr)^{d\alpha_d}} =
c_d\int_2^\infty \rho^{-d\alpha_d}\,\d\rho<\infty,
\end{aligned}
\]
since $2\alpha_2=4/3$, while $d\alpha_d=d/2>1$ when $d\geq3$.  Therefore, $u\in L^d(\mR^d;\mR^d)$. The field is also divergence-free in the distributional sense. In fact, on every annulus not meeting the origin,
\[
\div\bigl(\beta(|x|)Jx\bigr)
=\beta'(|x|)\frac{x\cdot Jx}{|x|}
+\beta(|x|){\rm tr}J=0.
\]
Moreover, $Jx$ is tangent to every sphere centered at the origin.  Hence, for $0<\varepsilon<\e^{-2}$ and $\varphi\in C_c^\infty(\mR^d)$, integration by parts on $\{\varepsilon<|x|<\e^{-2}\}$ yields that
\[
\int_{\{\varepsilon<|x|<\e^{-2}\}}
u(x)\cdot\nabla\varphi(x)\,\d x=0,\quad \int_{\{|x|>\e^{-2}\}}
u(x)\cdot\nabla\varphi(x)\,\d x=0.
\]
Here we used the fact that the flux of $u$ through boundary spheres vanishes.  Since $u\in L^1_{{\rm loc}}$, letting $\varepsilon\downarrow0$, we obtain
\[
\int u(x)\cdot\nabla\varphi(x)\,\d x=0.
\]
Consequently, $\div u=0$ in $\cD'(\mR^d)$.

By Lemma \ref{lem:hodge}, this drift admits a skew-symmetric stream matrix in ${\rm BMO}$, so $u\in{\rm BMO}^{-1}$. Taking $v=u$, $F=|u|$ and $a_\sigma=b_\sigma=0$ verifies Assumption \ref{ass:decomp}. The stated locally uniform Varadhan formula follows from Corollary \ref{cor:classes}, while positivity and joint continuity of the kernel follow from Theorem \ref{thm:feller}. It remains to prove the path-law assertion.

Let $\Omega:=C([0,1];\mR^d)$, let $\xi_t(\omega):=\omega(t)$ be the canonical process, and put $\cF_t^0:=\sigma(\xi_s:0\leq s\leq t)$. We use the following finite-entropy form of the converse Girsanov theorem to show that a common part of the diffusion and Brownian laws would force the drift energy to be finite; see \cite{Follmer1985entropy}.

\begin{lemma}
\label{lem:follmer}
Let $R$ and $Q$ be probability measures on $\Omega$.  Suppose that $M$ is an $R$-Brownian motion with respect to the usual $R$-augmentation of $(\cF_t^0)$, $Q\ll R$, and
\[
{\rm Ent}(Q\mid R)
:=
\int_\Omega\log\!\Big(\frac{\d Q}{\d R}\Big)\,\d Q<\infty.
\]
Then there is a predictable process $h$ such that
\[
\bE_Q\int_0^1|h_s|^2\,\d s<\infty
\]
and $M_t-\int_0^t h_s\,\d s$ is a $Q$-Brownian motion.
\end{lemma}

\begin{proof}[Proof of Theorem \ref{thm:path}]
Let $\mP$ be the law on $\Omega$ of any weak solution to \eqref{eq:sde} starting from the origin, and let $\mW_x$ be the law on $\Omega$ of $\sqrt2W$ starting from $x$.  We show that $\mP\perp\mW_0$.

Define
\[
E:=
\Big\{\omega\in\Omega:
\int_0^1|u(\xi_t(\omega))|^2\,\d t=\infty
\Big\}.
\]
We claim that
\begin{equation}\label{eq:we}
\mW_0(E)=1.
\end{equation}
\noindent {\em Step 1: The drift energy along Brownian paths.}

Let
\[
\tau_*:=\inf\{t\geq0:\xi_t\notin B_{\e^{-1}}(0)\}.
\]
For $N>2$, set
\[
\begin{aligned}
f_N(x)
&:=\mathbf1_{\{\e^{-N}<|x|<\e^{-2}\}}|u(x)|^2=\frac{\mathbf1_{\{\e^{-N}<|x|<\e^{-2}\}}|Jx|^2}
{|x|^4\bigl(\log(1/|x|)\bigr)^{2\alpha_d}},\\
\bar f_N(x)
&:=\frac{\mathbf1_{\{\e^{-N}<|x|<\e^{-2}\}}}
{|x|^2\bigl(\log(1/|x|)\bigr)^{2\alpha_d}}.
\end{aligned}
\]
Then $0\leq f_N\leq\bar f_N$.  Let $G_{B_{\e^{-1}}(0)}$ be the Green function of the killed Brownian motion $\xi$ in $B_{\e^{-1}}(0)$.  Define
\[
\begin{aligned}
V_N(x)&:=\bE_{\mW_x}\int_0^{\tau_*}f_N(\xi_t)\,\d t
=\int_{B_{\e^{-1}}(0)}G_{B_{\e^{-1}}(0)}(x,y)f_N(y)\,\d y,\\
\bar V_N(x)&:=\bE_{\mW_x}\int_0^{\tau_*}\bar f_N(\xi_t)\,\d t
=\int_{B_{\e^{-1}}(0)}G_{B_{\e^{-1}}(0)}(x,y)\bar f_N(y)\,\d y.
\end{aligned}
\]
Recall that
\begin{equation}\label{eq:green}
G_{B_{\e^{-1}}(0)}(0,y)=
\begin{cases}
\displaystyle\frac1{2\pi}\log\frac{\e^{-1}}{|y|},&d=2,\\[1.2ex]
\displaystyle\frac{|y|^{2-d}-\e^{d-2}}
{(d-2)|\mS^{d-1}|},&d\geq3,
\end{cases}
\end{equation}
see, for instance, \cite[Chapter 3]{yuval2010brownian}. Since $G_{B_{\e^{-1}}(0)}(0,y)$ is radial and
\[
\frac1{|\mS^{d-1}|}
\int_{\mS^{d-1}}|J\theta|^2\,\d\theta
=\frac1{|\mS^{d-1}|}
\int_{\mS^{d-1}}(\theta_1^2+\theta_2^2)\,\d\theta
=\frac2d,
\]
we have
\[
m_N:=V_N(0)=\frac2d\bar V_N(0).
\]
Using \eqref{eq:green}, we see
\[
m_N=
\begin{cases}
\displaystyle\int_2^N\frac{\rho-1}{\rho^{4/3}}\,\d\rho,&d=2,\\[2ex]
\displaystyle\frac2{d(d-2)}\int_2^N
\frac{1-\e^{-(d-2)(\rho-1)}}{\rho}\,\d\rho,&d\geq3.
\end{cases}
\]
Thus, $m_N<\infty$ for every $N$ and $m_N\to\infty$ as $N\to\infty$. Since $\bar f_N$ is radial, the equation $-\Delta\bar V_N=\bar f_N$ in $B_{\e^{-1}}(0)$ and $\bar V_N=0$ on $\partial B_{\e^{-1}}(0)$ can be rewritten as
\[
\bigl(r^{d-1}\bar V_N'(r)\bigr)'
=-r^{d-1}\bar f_N(re_1),
\quad 0<r<\e^{-1}.
\]
Since $\bar f_N$ vanishes in $B_{\e^{-N}}(0)$, the mean-value property implies $\bar V_N'(r)=0$ for $0<r<\e^{-N}$.  Integrating the radial equation from $0$ to $r$, we therefore obtain
\[
\bar V_N'(r)
=-\frac1{r^{d-1}}\int_0^r
s^{d-1}\bar f_N(se_1)\,\d s\leq0.
\]
It follows that
\[
0\leq V_N\leq\bar V_N,
\quad
\sup_{B_{\e^{-1}}(0)}V_N\leq\bar V_N(0)=\frac d2m_N.
\]
Put
\[
A_N:=\int_0^{\tau_*}f_N(\xi_t)\,\d t.
\]
Then
\[
A_N^2
=2\int_0^{\tau_*}f_N(\xi_s)
\Big(\int_s^{\tau_*}f_N(\xi_t)\,\d t\Big)\d s.
\]
By the Markov property and Fubini's theorem, we get
\[
\begin{aligned}
\bE_{\mW_0}A_N^2
&=2\bE_{\mW_0}\int_0^{\tau_*}
f_N(\xi_t)V_N(\xi_t)\,\d t\\
&\leq2\sup_{B_{\e^{-1}}(0)}V_N\,\bE_{\mW_0}A_N
\leq d m_N^2,
\end{aligned}
\]
where we used the fact that $\bE_{\mW_0}A_N=V_N(0)=m_N$.  By the Paley--Zygmund inequality \cite[Exercise 3.5]{yuval2010brownian}, we then obtain
\[
\mW_0\Big(
A_N\geq\frac{m_N}{2}
\Big)
\geq\frac{(\bE_{\mW_0}A_N)^2}
{4\bE_{\mW_0}A_N^2}
\geq\frac1{4d}.
\]
Since $f_N\uparrow|u|^2$ away from the origin and $m_N\to\infty$, we obtain
\[
\mW_0\Big(\int_0^{\tau_*}|u(\xi_t)|^2\,\d t=\infty\Big)
\geq\frac1{4d}.
\]
Here $\tau_*<\infty$ almost surely, and Brownian motion in dimension $d\geq2$, started from the origin, almost surely does not return to the point $0$ at a positive time.  Thus, $|u(\xi_t)|$ is bounded on every path interval $[\varepsilon,\tau_*]$ with $\varepsilon<\tau_*$.  On this event, it follows that
\[
\int_0^\varepsilon|u(\xi_t)|^2\,\d t=\infty
\quad\text{for every }\varepsilon>0.
\]

Therefore, the event
\[
\bigcap_{m=1}^\infty
\Big\{\int_0^{1/m}|u(\xi_t)|^2\,\d t=\infty\Big\}
\]
has positive probability and belongs to the Brownian germ sigma-field $\cF_{0+}$.  By Blumenthal's zero--one law, its probability is $1$.  This proves \eqref{eq:we} for every $d\geq2$.

\medskip

\noindent \textit{Step 2: Mutual singularity of the path law.}

Suppose, for contradiction, that $\mP\not\perp\mW_0$.  Set $S:=\mP+\mW_0$ and
\[
p:=\frac{\d\mP}{\d S},
\quad
q:=\frac{\d\mW_0}{\d S}.
\]
Then
\[
m:=\int_\Omega(p\wedge q)\,\d S>0.
\]
Define a probability measure $Q$ by
\[
\frac{\d Q}{\d S}:=\frac{p\wedge q}{m}.
\]
On the sets where the denominators are positive,
\[
\frac{\d Q}{\d\mP}
=
\frac{p\wedge q}{mp}\leq\frac1m,
\quad
\frac{\d Q}{\d\mW_0}
=
\frac{p\wedge q}{mq}\leq\frac1m.
\]
Hence $Q\ll\mP$, $Q\ll\mW_0$, and
\[
{\rm Ent}(Q\mid\mP)
\vee
{\rm Ent}(Q\mid\mW_0)
\leq\log(1/m)<\infty.
\]

Apply Lemma~\ref{lem:follmer} first with $R=\mW_0$ and $M=\xi/\sqrt2$. There is a predictable process $h$ such that
\[
\bE_Q\int_0^1|h_s|^2\,\d s<\infty
\]
and
\[
B_t^Q:=\frac{\xi_t}{\sqrt2}-\int_0^t h_s\,\d s
\]
is a standard $Q$-Brownian motion.

On the other hand, since $\mP$ is a weak-solution law,
\[
M_t^\mP
:=
\frac1{\sqrt2}\Big(\xi_t-\int_0^t u(\xi_s)\,\d s\Big)
\]
is a standard Brownian motion under $\mP$ with respect to the usual $\mP$-augmentation of $(\cF_t^0)$. Applying Lemma~\ref{lem:follmer} with $R=\mP$, we obtain a predictable process $k$ such that
\[
\bE_Q\int_0^1|k_s|^2\,\d s<\infty
\]
and
\[
\widetilde B_t^Q
:=
M_t^\mP-\int_0^t k_s\,\d s
\]
is a standard $Q$-Brownian motion.  Therefore, in the same usual filtration under $Q$,
\[
\xi_t
=
\sqrt2\int_0^t h_s\,\d s+\sqrt2B_t^Q
=
\int_0^t\bigl(u(\xi_s)+\sqrt2k_s\bigr)\,\d s
+\sqrt2\widetilde B_t^Q.
\]
Uniqueness of the continuous semimartingale decomposition yields
\[
u(\xi_s)=\sqrt2\,(h_s-k_s)
\quad \d s\otimes\d Q\text{-a.e.}
\]
Consequently,
\[
\int_0^1|u(\xi_s)|^2\,\d s
\leq
4\int_0^1|h_s|^2\,\d s
+
4\int_0^1|k_s|^2\,\d s
<\infty
\quad Q\text{-a.s.},
\]
so $Q(E)=0$.  But $Q\ll\mW_0$ and \eqref{eq:we} imply $Q(E)=1$, a contradiction.
\end{proof}

\begin{appendix}

\renewcommand{\theequation}{\thesection.\arabic{equation}}
\section{\texorpdfstring{${\rm BMO}^{-1}$}{BMO inverse}, critical embeddings, and
Lipschitz weights}\label{sec:aux}

We collect the definitions and auxiliary results used in the main proofs: the divergence-free Hodge representation, the critical embeddings, the localized Hardy--BMO and trace estimates, and the construction of the flattened weights.

\medskip

For $f\in L^1_{{\rm loc}}(\mR^d)$ and a Euclidean ball $B_r(x)$, write $f_{B_r(x)}:=\fint_{B_r(x)} f(y)\,\d y$.  The space of functions of bounded mean oscillation is
\[
{\rm BMO}(\mR^d)
:=
\Big\{f\in L^1_{{\rm loc}}(\mR^d):
\|f\|_{{\rm BMO}}
:=
\sup_{x\in\mR^d,\,r>0}\fint_{B_r(x)}|f(y)-f_{B_r(x)}|\,\d y<\infty
\Big\},
\]
where functions differing by a constant are identified. For vector- and matrix-valued functions the seminorm is defined componentwise, with the Euclidean combination of the component seminorms.

\begin{definition}
A scalar distribution $f\in\cD'(\mR^d)$ belongs to ${\rm BMO}^{-1}(\mR^d)$ if $f=\div F$ for some vector field $F=(F_1,\ldots,F_d)$ with entries in ${\rm BMO}$. We set
\[
\|f\|_{{\rm BMO}^{-1}}
:=
\inf_{f=\div F}
\Big(\sum_{i=1}^d\|F_i\|_{{\rm BMO}}^2\Big)^{1/2}.
\]
All  ${\rm BMO}$  quantities are seminorms modulo constants.
\end{definition}

Let $d\geq2$, and let $R_k$ denote the $k$-th Riesz transform, with Fourier multiplier $\mathrm i\xi_k/|\xi|$.

For a vector distribution $a=(a_1,\ldots,a_d)$, we use the Euclidean combination of the component seminorms.  If $\div a=0$, the Hodge projection and the boundedness of Riesz transforms on  ${\rm BMO}$ imply that
\[
\|a\|_{{\rm BMO}^{-1}}
\asymp_d
\inf_{\substack{a=\div A\\ A^{\mathsf T}=-A}}
\|A\|_{{\rm BMO}}.
\]
After fixing the gauge by zero average on $B_1(0)$, the realizing skew potential depends linearly on $a$.

For a measurable scalar or vector field $F$, we use the convention
\[
\|F\|_{L^{d,\infty}}
:=\sup_{\lambda>0}
\lambda\bigl|\{x\in\mR^d:|F(x)|>\lambda\}\bigr|^{1/d}.
\]

The following Hodge representation connects divergence-free weak-$L^d$ drifts with the ${\rm BMO}$ stream framework used throughout the paper.

\begin{lemma}\label{lem:hodge}
Let $u\in L^{d,\infty}(\mR^d;\mR^d)$ be divergence-free.  Then there is a skew-symmetric $U\in{\rm BMO}(\mR^d;\mR^{d\times d})$, unique modulo constant skew-symmetric matrices, such that $u=\div U$ and
\begin{equation*}
\partial_kU_{ij}=R_kR_j u_i-R_kR_i u_j,
\quad 1\leq i,j,k\leq d.
\end{equation*}
Moreover,
\[
\sum_{i<j}\|U_{ij}\|_{{\rm BMO}}^2
\leq C\|u\|_{L^{d,\infty}}^2.
\]
\end{lemma}

\begin{proof}
For $i<j$, set
\[
F_{kij}:=R_kR_j u_i-R_kR_i u_j.
\]
Then $F_{kij}=-F_{kji}$, $\partial_\ell F_{kij}-\partial_kF_{\ell ij}$=0, i.e. $\d\sum_{k} F_{kij}\d x_k=0$, for each $i,j\in \{1,2,\cdots,d\}$, and Riesz-transform boundedness on Lorentz spaces implies $F_{\cdot ij}\in L^{d,\infty}$. By the distributional Poincar\'e lemma, there exists a locally integrable $U_{ij}$, unique up to a constant, with $\partial_kU_{ij}=F_{kij}$. In particular, $U_{ij}\in W^{1,q}_{\rm loc}$ for every $q<d$.  Set $U_{ji}:=-U_{ij}$ and $U_{ii}:=0$.  For $D_j:=\sum_i\partial_iU_{ij}-u_j$,
\[
\partial_\ell D_j
=R_\ell R_j\div u
-\sum_i\partial_iR_\ell R_i u_j-\partial_\ell u_j=0.
\]
Thus, $D_j$ is constant.  Since $D_j\in L^{d,\infty}(\mR^d)$, it vanishes, and $u=\div U$.

Since Riesz transforms are bounded on $L^{d,\infty}$,
\[
\|\nabla U\|_{L^{d,\infty}}
\leq C\|u\|_{L^{d,\infty}}.
\]
For a ball $B_r(x)$, the Poincar\'e inequality and H\"older's inequality in Lorentz spaces yield
\[
\fint_{B_r(x)}|U-U_{B_r(x)}|
\leq Cr\fint_{B_r(x)}|\nabla U|
\leq C\|\nabla U\|_{L^{d,\infty}}
\leq C\|u\|_{L^{d,\infty}}.
\]
Taking the supremum over $x\in\mR^d$ and $r>0$ proves the  ${\rm BMO}$  estimate.
\end{proof}

The next estimate controls the small ${\rm BMO}^{-1}$ part against the Lipschitz weight chosen for the weak-$L^d$ envelope in Lemma \ref{lem:weights}.

\begin{lemma}[A localized Hardy--BMO estimate]
\label{lem:clms}
Let $A^{\mathsf T}=-A$ belong componentwise to ${\rm BMO}(\mR^d)$, and let $\chi$ be a real bounded Lipschitz function. Then, for every real $\varphi\in C_c^\infty(\mR^d)$,
\begin{equation}\label{eq:clms}
2\left|\cE_A^a(\chi\varphi,\varphi)\right|
\leq C\|A\|_{{\rm BMO}}\|\nabla\chi\|_\infty
\|\nabla\varphi\|_2\|\varphi\|_2,
\end{equation}
where $\cE_A^a$ is the antisymmetric form defined in \eqref{eq:local-forms}.
\end{lemma}

\begin{proof}
The estimate \eqref{eq:hbmo} verifies the hypothesis of Lemma \ref{lem:cutoff} with $M=C\|A\|_{{\rm BMO}}$ and $b=0$. Its conclusion is exactly \eqref{eq:clms}.
\end{proof}

We first record the localized form of the Meyers--Ziemer trace inequality used in the weight construction.

\begin{lemma}[A localized Meyers--Ziemer trace inequality]
\label{lem:meyers-ziemer}
Let $d\geq2$, let $\mu$ be a nonnegative Radon measure on $\mR^d$, and let $A,R>0$. Suppose that
\[
\mu(B_r(x))\leq A r^{d-1}
\qquad\text{for every }x\in\mR^d\text{ and }0<r\leq R.
\]
Then, for every real $\varphi\in C_c^\infty(\mR^d)$,
\begin{equation}\label{eq:mz-local}
\int_{\mR^d}\varphi^2\,\d\mu
\leq CA\|\nabla\varphi\|_2\|\varphi\|_2
+C\frac{A}{R}\|\varphi\|_2^2,
\end{equation}
where $C$ depends only on $d$.
\end{lemma}

\begin{proof}
The classical Meyers--Ziemer inequality \cite[Theorem~4.7]{MeyersZiemer1977} states that, if a nonnegative Radon measure $\nu$ satisfies $\nu(B_r(x))\leq K r^{d-1}$ for every $x\in\mR^d$ and $r>0$, then
\[
\int_{\mR^d}|h|\,\d\nu
\leq CK\int_{\mR^d}|\nabla h|\,\d x,
\qquad h\in C_c^\infty(\mR^d).
\]
Partition $\mR^d$ into half-open cubes $Q$ of diameter $R$ and set $\mu_Q:=\mu|_Q$. Since $\mu(Q)\leq A R^{d-1}$, the hypothesis implies $\mu_Q(B_r(x))\leq A r^{d-1}$ for all $r>0$: for $r\leq R$ use $\mu_Q\leq\mu$, and for $r>R$ use the total mass. Choose smooth cutoffs $0\leq\chi_Q\leq1$, equal to $1$ on $\overline Q$, supported in the doubled cube $2Q$, and satisfying $\|\nabla\chi_Q\|_\infty\leq C/R$. Applying the above inequality to $\mu_Q$ and $h=\chi_Q\varphi^2$, and then summing, we obtain
\[
\int\varphi^2\,\d\mu
\leq CA\sum_Q\int|\nabla(\chi_Q\varphi^2)|\,\d x 
\leq CA\int|\varphi||\nabla\varphi|\,\d x
+C\frac{A}{R}\int\varphi^2\,\d x.
\]
Cauchy--Schwarz yields \eqref{eq:mz-local}. All constants $C$ depend only on $d$.
\end{proof}

To control the weak-$L^d$ part in Lemma \ref{lem:weights}, we now construct a Lipschitz weight that is flat on an exceptional set and nearly preserves the directional increment. The exceptional set is independent of the endpoints.

\begin{lemma}
\label{lem:flatten}
Let $d\geq2$, let $F\geq0$ belong to $L^{d,\infty}(\mR^d)$, and $N:=\|F\|_{L^{d,\infty}}$.  For every $\eta,R>0$, there is an open set $O_{\eta,R}$ such that
\begin{equation}\label{eq:wl-trace}
\int_{O_{\eta,R}^c}F\varphi^2\,\d x
\leq C_{\CFzero}\eta\|\nabla\varphi\|_2\|\varphi\|_2
+C_{\CFzero}\frac{\eta}{R}\|\varphi\|_2^2
\end{equation}
for every real $\varphi\in C_c^\infty(\mR^d)$.  Moreover, for every $e\in\mS^{d-1}$ and $x,y\in\mR^d$, there is a real bounded Lipschitz function $\psi_{e,x,y}$ such that
\begin{equation}\label{eq:wl-weight}
\|\nabla\psi_{e,x,y}\|_\infty\leq1,
\quad
\nabla\psi_{e,x,y}=0\quad\hbox{a.e. on }O_{\eta,R},
\end{equation}
and
\begin{equation}\label{eq:wl-cal}
\left|
[\psi_{e,x,y}(y)-\psi_{e,x,y}(x)]-e\cdot(y-x)
\right|
\leq C_{\Cdist}\Big(\frac{N}{\eta}\Big)^dR.
\end{equation}
The set $O_{\eta,R}$ depend only on $d,F,\eta,R$. With this choice fixed, $\psi_{e,x,y}$ can be chosen to depend only on $d,F,\eta,R,e,x,y$. The constants $C_{\Cdist}$ and $C_{\CFzero}$ depend only on $d$.
\end{lemma}

\begin{proof}
Let $\eta, R>0$ fixed. We call a ball $B_r(z)$ critical if $0<r\leq R$ and
\begin{equation}\label{eq:critical_ball}
  \int_{B_r(z)}F\,\d x>\eta r^{d-1}.
\end{equation}
By Vitali's covering lemma, there is a countable disjoint family of critical balls $\{B_{r_j}(z_j)\}_{j\in \cJ}$ such that the union of all critical balls is contained in
\[
O_{\eta,R}:=\bigcup_{j\in\cJ}B_{5r_j}(z_j).
\]
Fix this family of balls depending only $d,F,\eta,R$, before choosing $e,x,y$ or $\varphi$.
By the weak H\"older inequality, we have
\begin{equation}\label{eq:wl-set}
\int_EF\,\d x\leq CN|E|^{1-1/d}
\end{equation}
for every measurable set $E$ of finite measure.  For $k\geq0$, put 
\[
\cJ_k:=\{j:2^{-k-1}R<r_j\leq2^{-k}R\}.  
\]
\eqref{eq:critical_ball}, disjointness of $\{B_{r_j}(z_j)\}$ and \eqref{eq:wl-set} imply 
\[
\eta \#\cJ_k(2^{-k-1}R)^{d-1}
<\int_{\bigcup_{j\in \cJ_k}B_{r_j}(z_j)}F\,\d x
\leq CN\bigl(\#\cJ_k(2^{-k}R)^d\bigr)^{1-1/d}.
\]
Hence $\#\cJ_k\leq C(N/\eta)^d$, since every finite subfamily satisfies this bound.  Summing over $k$ yields
\begin{equation}\label{eq:radii}
\sum_jr_j
\leq\sum_{k\geq0}2^{-k}R\cdot\#\cJ_k
\leq C\Big(\frac{N}{\eta}\Big)^dR.
\end{equation}

Let $\mu:=F\mathbf1_{O_{\eta,R}^c}\,\d x$.  If $0<r\leq R/4$ and $\mu(B_r(z))>0$, choose $z'\in B_r(z)\cap O_{\eta,R}^c$.  The ball $B_{2r}(z')$ is not critical, and hence
\[
\mu(B_r(z))\leq C\eta r^{d-1}.
\]
The same bound is immediate when $\mu(B_r(z))=0$. Applying Lemma \ref{lem:meyers-ziemer} with growth constant $A=C\eta$ and radius $R/4$ proves \eqref{eq:wl-trace}, with $C_{\CFzero}$ depending only on $d$.

For the construction of weight $\psi$, we set $a:=\min\{e\cdot x,e\cdot y\}$, $b:=\max\{e\cdot x,e\cdot y\}$, and
\[
A_e:=\bigcup_{j\in \cJ}(e\cdot z_j-5r_j,e\cdot z_j+5r_j),
\]
\[
\psi_{e,x,y}(z):=\int_a^{e\cdot z}
\mathbf1_{[a,b]\setminus A_e}(s)\,\d s.
\]
Then $0\leq\psi_{e,x,y}\leq b-a$ and $\|\nabla\psi_{e,x,y}\|_\infty\leq1$.  Its gradient vanishes on $O_{\eta,R}$, since every $B_{5r_j}(z_j)$ projects into $A_e$.  Moreover,
\[
\left|
[\psi_{e,x,y}(y)-\psi_{e,x,y}(x)]-e\cdot(y-x)
\right|
\leq |A_e|\leq10\sum_jr_j.
\]
Together with \eqref{eq:radii}, this proves \eqref{eq:wl-weight} and \eqref{eq:wl-cal}.
\end{proof}

% Figure \ref{fig:flattened-weight} illustrates the construction of the weight in Lemma \ref{lem:flatten}.

% \begin{figure}[htbp]
% \centering
% \begingroup
% \setlength{\fboxsep}{6pt}
% \colorbox{white}{\color{black}\input{figures/flattened-weight.tex}}
% \endgroup
% \caption{Geometry of the flattened weight. The disjoint blue critical balls are enlarged by five to form $O_{\eta,R}$. Their projections merge into the orange set $A_e$. The profile $H$ has slope $1$ on $[a,b]\setminus A_e$ and is constant elsewhere, so $\nabla\psi_{e,x,y}=0$ almost everywhere on $O_{\eta,R}$. The lost endpoint increment is the length removed from $[a,b]$.}
% \label{fig:flattened-weight}
% \end{figure}

\section{The BMO heat kernel and its estimates}\label{sec:approx}
We record the kernel construction for jointly Borel skew streams in $L^\infty_t{\rm BMO}_x$ and justify the BMO consequences used in the main results. This appendix retains a uniform BMO bound throughout the kernel approximation.

For a jointly Borel skew stream $U\in L^\infty_t{\rm BMO}_x$, set
\[
N_U:=\esssup_{t\geq0}
\Big(\sum_{i<j}\|U_{ij}(t)\|_{{\rm BMO}}^2\Big)^{1/2}.
\]
Normalize $U$ by $U_{B_1(0)}(t)=0$ and extend it by zero to $t<0$. A smooth approximation is a sequence of skew matrices $U_n\in C_b^\infty(\mR_+\times\mR^d;\mR^{d\times d})$ satisfying
\begin{equation}\label{eq:approx}
N_*:=\sup_n\sup_{t\in\mR}
\Big(\sum_{i<j}\|(U_n)_{ij}(t)\|_{{\rm BMO}}^2\Big)^{1/2}<\infty,
\quad
U_n\to U
\quad\text{in }L^2_{{\rm loc}}(\mR\times\mR^d).
\end{equation}
\begin{theorem}\label{thm:feller}
Let $U\in L^\infty_t{\rm BMO}_x$ be jointly Borel and skew-symmetric, and let $u=\div U$. Then:
\begin{enumerate}[(a)]
\item There is a unique approximation-independent kernel $p$.  It is positive, jointly continuous, and locally H\"older on $\{s<t\}$, and satisfies \eqref{eq:aronson}.  It defines a conservative, Lebesgue-preserving Feller evolution $(P_{s,t})$, which determines a continuous Feller family $(\mP_{s,x})$ from every starting point.
\item For every sequence $(U_n)$ satisfying \eqref{eq:approx},
\[
p^n\to p\quad\text{locally uniformly on }\{s<t\}.
\]
If $s_n\to s\in\mR$, $x_n\to x$, and $0<T<\infty$, then
\[
\mP^n_{s_n,x_n}\circ(X^n_{s_n+\,\cdot})^{-1}
\Longrightarrow
\mP_{s,x}\circ(X_{s+\,\cdot})^{-1}
\quad\text{in }C([0,T];\mR^d).
\]
\end{enumerate}
\end{theorem}

For later use, recall the Hardy--BMO definition of the antisymmetric form for $f,g\in H^1(\mR^d)$:
\begin{equation}\label{eq:bmo}
\cE^a_t(f,g)
=\frac12\sum_{i,j=1}^d
{}_{{\rm BMO}}\!\left\langle U_{ij}(t),
\partial_jf\,\partial_i g-\partial_if\,\partial_jg
\right\rangle_{\cH^1},
\end{equation}
where $\cH^1$ is the real Hardy space. By the Coifman--Lions--Meyer--Semmes estimate \cite{CLMS93compensated}, we obtain
\begin{equation}\label{eq:hbmo}
|\cE^a_t(f,g)|
\leq CN_U\|\nabla f\|_2\|\nabla g\|_2.
\end{equation}
The form agrees with \eqref{eq:local-forms} on compact smooth tests and is unchanged by adding a spatially constant skew matrix to $U(t)$.

\subsection{Kernel approximation and Feller process}
\begin{proof}[Proof of Theorem \ref{thm:feller}]
Use the normalization and negative-time extension stated before Theorem \ref{thm:feller}. John--Nirenberg implies $U\in L^2_{{\rm loc}}(\mR\times\mR^d)$. Such approximations exist.  Let $\vartheta_n\in C_c^\infty(0,n^{-1})$ and $\rho\in C_c^\infty(B_1(0))$ be nonnegative unit-mass mollifiers, write $\rho_r(x):=r^{-d}\rho(x/r)$, and choose $0<\varepsilon_n<R_n$, $\varepsilon_n\downarrow0$, $R_n\uparrow\infty$, with $CN_U/R_n\leq n^{-1}$.  Define
\begin{equation}\label{eq:bandpass}
\begin{aligned}
F_n(t,x)&:=\int_0^{1/n}\vartheta_n(r)
 \bigl[(\rho_{\varepsilon_n}-\rho_{R_n})*U(t-r)\bigr](x)\,\d r,\\
U_n(t,x)&:=F_n(t,x)-\fint_{B_1(0)}F_n(t,z)\,\d z.
\end{aligned}
\end{equation}
For $0<\varepsilon<R$, the standard estimates (see, for instance \cite{grafakos2009modern})
\[
\|(\rho_\varepsilon-\rho_R)*b\|_\infty
\leq C(1+\log(R/\varepsilon))\|b\|_{{\rm BMO}},
\]
\[
\|\rho_r*b\|_{{\rm BMO}}\leq\|b\|_{{\rm BMO}},
\quad
\|\nabla(\rho_R*b)\|_\infty\leq CR^{-1}\|b\|_{{\rm BMO}}
\]
show that $U_n\in C_b^\infty$, $\sup_n\|U_n\|_{L^\infty_t{\rm BMO}_x}\leq CN_U$, and $U_n\to U$ in $L^2_{{\rm loc}}$.  Thus, sequences satisfying \eqref{eq:approx} exist.

For the kernel and Feller conclusions, we use the Gaussian-kernel argument discussed in \cite[Remark~6.3(2)]{KV25critical}. We briefly record its application to the present time-dependent setting, including $d=2$ and arbitrary sequences satisfying \eqref{eq:approx}. By \cite[Theorems~2.1 and~2.2]{QX19parabolic2}, their classical kernels $p^n$ satisfy \eqref{eq:aronson} and local H\"older estimates uniformly in $n$. Apply these results to the forward matrices $I-U_n$ and their time-reversed transposes to obtain local equicontinuity in both time variables. Thus, subsequential limits exist locally uniformly on $\{s<t\}$.

Let $\cE_t^n$ be the form corresponding to $U_n$. For $s<T$ and $f\in C_c^\infty$, put $v_n(t,y)=\int f(x)p^n_{s,t}(x,y)\,\d x$. The energy bounds and \eqref{eq:hbmo} bound $v_n$ in $L^\infty(s,T;L^2)\cap L^2(s,T;H^1)$ and $\partial_tv_n$ in $L^2(s,T;H^{-1})$. For $\phi\in C_c^\infty([s,T)\times\mR^d)$, the forward weak equation is
\[
-\int_s^T(v_n,\partial_t\phi)\,\d t
+\int_s^T\cE_t^n(\phi,v_n)\,\d t=(f,\phi(s)).
\]
Its coefficient error is bounded by
\[
\|\nabla\phi\|_\infty
\|U_n-U\|_{L^2({\rm supp}\phi)}
\|\nabla v_n\|_{L^2((s,T)\times\mR^d)}\to0.
\]
Consequently, every limit has initial trace $f$ and solves the same forward energy equation. Energy uniqueness \cite[Section~5, in particular Theorem~5.5]{QX19parabolic2} identifies all limits. By kernel continuity and interleaving of two approximation sequences, we obtain one pointwise kernel $p$, independent of the approximation, with $p^n\to p$ locally uniformly. Time reversal also identifies $P_{r,t}g$ as the unique backward energy solution with terminal data $g$; in particular, this evolution agrees with the classical one on smooth time intervals.

The Gaussian tails pass both mass identities and the Chapman--Kolmogorov identity to $p$. They also imply preservation of $C_0$ and strong continuity at the time diagonal. Hence $p$ defines a conservative, Lebesgue-preserving Feller evolution, with positivity and local H\"older continuity inherited from the uniform kernel bounds. As in \cite[Remark~6.3(2)]{KV25critical}, by the Gaussian bound and the Markov property, we have, for $q>2$,
\begin{equation}\label{eq:moment}
\mE^n_{s,x}|X^n_{s+t}-X^n_{s+r}|^q
+\mE_{s,x}|X_{s+t}-X_{s+r}|^q
\leq C|t-r|^{q/2},\quad 0\leq r<t\leq T,
\end{equation}
where $C$ depends only on $d,q,N_U,N_*$, uniformly in $n,s,x$. The consistent finite-dimensional laws therefore have continuous versions from every starting point. Local uniform kernel convergence and Gaussian tails imply convergence of the finite-dimensional laws when $(s_n,x_n)\to(s,x)$; \eqref{eq:moment} implies tightness on $C([0,T];\mR^d)$. This proves the stated path-law convergence and completes the proof.
\end{proof}

\subsection{Uniform estimates for the approximations}
\begin{proposition}\label{prop:bounds}
Let $U$ satisfy the hypotheses of Theorem \ref{thm:feller}.
\begin{enumerate}[(i)]
\item Assumption \ref{ass:sector} holds with $M=CN_U$ and $m=0$. It holds uniformly along every sequence satisfying \eqref{eq:approx}, with $M=CN_*$ and $m=0$. If specified data $M,m$ are available, the approximations in \eqref{eq:bandpass} satisfy \eqref{eq:sector} with $2M$ and $m_n=2\vartheta_n*m$, where $m$ is extended by zero outside $(0,T)$. Thus, $\|m_n\|_{L^2(0,T)}\leq2\|m\|_{L^2(0,T)}$.
\item If Assumption \ref{ass:decomp} holds, the approximations in \eqref{eq:bandpass} preserve its decomposition with a common weak-$L^d$ envelope and uniformly vanishing short-time remainder integrals. Their kernels satisfy \eqref{eq:weighted-opt} with $e_\delta$ defined using one common integrated remainder $c_\delta\in L^1(0,T)$, independent of $n,x,y$, in \eqref{eq:upper-data}.
\end{enumerate}
\end{proposition}

\begin{proof}
By \eqref{eq:hbmo}, condition \eqref{eq:sector} holds with $M=CN_U$ and $m=0$. The same argument applies uniformly to the approximations in \eqref{eq:approx}.

We also check preservation of specified $M,m$ under \eqref{eq:bandpass}. Spatial translation preserves all norms in \eqref{eq:sector}. Averaging against a nonnegative unit-mass spatial mollifier preserves $M,m$. Time convolution replaces $m$ by $\vartheta_n*m$. The difference of the two spatial convolutions therefore has parameters $2M$ and $m_n=2\vartheta_n*m$. Normalization by a spatially constant skew matrix does not change the form. Since the time mollifier uses only earlier times and $U=0$ at negative times, this applies on $(0,T)$. Jensen's inequality and the same one-sided support imply
\[
\int_0^h m_n(t)^2\,\d t
\leq4\int_0^h m(t)^2\,\d t,\quad 0<h\leq T.
\]
Thus, both parameters in the lower estimate remain uniformly bounded.

Extend the three summands in \eqref{eq:decomp} by zero to negative times, and put $\beta_\sigma(t):=\|b_\sigma(t)\|_\infty$. Since convolution commutes with divergence, their drifts are
\[
u_n=\vartheta_n*(\rho_{\varepsilon_n}-\rho_{R_n})*u.
\]
Here $\vartheta_n$ acts in time and the other convolutions act in space. Write
\[
v_n:=\vartheta_n*\rho_{\varepsilon_n}*v,\quad a_{\sigma,n}:=\vartheta_n*(\rho_{\varepsilon_n}-\rho_{R_n})*a_\sigma, 
\]
\[
b_{\sigma,n}:=\vartheta_n*(\rho_{\varepsilon_n}-\rho_{R_n})*b_\sigma
-\vartheta_n*\rho_{R_n}*v.
\]
Then $u_n=v_n+a_{\sigma,n}+b_{\sigma,n}$, and all three terms are divergence-free.  Let $\mathcal M F$ be the Hardy--Littlewood maximal function of $F$.  Positivity of the mollifiers implies
\[
|v_n(t,x)|\leq C\mathcal M F(x),\quad
\|\mathcal M F\|_{L^{d,\infty}}\leq C\|F\|_{L^{d,\infty}}.
\]
Convolution is bounded on ${\rm BMO}^{-1}$, so we may fix one constant $C\geq1$, large enough for the preceding maximal-function estimates and the following convolution bounds, such that
\[
\esssup_{0<t<T}\|a_{\sigma,n}(t)\|_{{\rm BMO}^{-1}}
\leq C\sigma.
\]
Finally, by the weak H\"older inequality, we have $\|\rho_{R_n}*v(t)\|_\infty\leq CR_n^{-1}\|F\|_{L^{d,\infty}}$, and hence
\[
\|b_{\sigma,n}(t)\|_\infty
\leq2(\vartheta_n*\beta_\sigma)(t)
+\frac{C\|F\|_{L^{d,\infty}}}{R_n}.
\]
The time convolution uses only earlier times.  Therefore, for $0<h<T$,
\[
\int_0^h\|b_{\sigma,n}(t)\|_\infty\,\d t
\leq2\int_0^h\beta_\sigma(t)\,\d t
+\frac{Ch\|F\|_{L^{d,\infty}}}{R_n}.
\]
Use the decomposition at parameter $\sigma_\delta/C$. Applying Lemma \ref{lem:weights} to the fixed envelope $C\mathcal M F$, we obtain weights independent of $n$ and remainders $K_\delta+\|b_{\sigma_\delta/C,n}(t)\|_\infty$. Choose $R_n\geq1$. Their integrals over $(0,h)$ are bounded by those of
\[
c_\delta(t):=K_\delta+2\beta_{\sigma_\delta/C}(t)
               +C\|F\|_{L^{d,\infty}}.
\]
From the weighted proof of Theorem \ref{thm:upper}, we therefore obtain \eqref{eq:weighted-opt} for every $n$ with this common $c_\delta$. In particular, no spatial convolution error needs to be added. The fixed constant $C$ depends only on $d$ and the chosen mollifiers, and is independent of $n,\sigma,\delta,h$.
\end{proof}

\end{appendix}

\section*{\bf Acknowledgments}
The authors would like to thank Zhen-Qing Chen, Xin Chen, Hirofumi Osada, and Zhongmin Qian for helpful discussions.

The research of Guohuan Zhao is supported by the National Key Research and Development Program of China (No. 2024YFA1013503) and the National Natural Science Foundation of China (Grant No. 12271352).

\medskip

\paragraph*{\bf Use of AI}

The author initially used \emph{ChatGPT-5.6} to find the example discussed in Theorem \ref{thm:path} for dimensions $d\geq 3$ and it provided an elegant proof based on the stationarity of the Ornstein-Uhlenbeck process; however, that argument could not be extended to the two-dimensional case. The author then generalized this result to all dimensions $d\geq 2$ and constructed an alternative argument to establish the singularity of the path law.

\small
\bibliographystyle{alphaabbrv}
\bibliography{mybib.bib}

\end{document}